\documentclass{amsart}

\usepackage{amssymb}

\newtheorem{thm}{Theorem}

\newtheorem{lem}[thm]{Lemma}
\newtheorem{cor}[thm]{Corollary}

\newtheorem{prop}[thm]{Proposition}

\theoremstyle{definition}
\newtheorem{defn}[thm]{Definition}

\newtheorem{say}[thm]{}
\newtheorem{exmp}[thm]{Example}
\newtheorem{exmps}[thm]{Examples}

\newtheorem{ques}[thm]{Question}    

\newtheorem{procedure}[thm]{Procedure}

\newtheorem*{ack}{Acknowledgments}      

\newtheorem{defn-thm}[thm]{Definition--Theorem}  
\newtheorem{defn-lem}[thm]{Definition--Lemma}  

\newtheorem{ass}[thm]{Assumption}

\theoremstyle{remark}
\newtheorem{claim}[thm]{Claim}

\let \cedilla =\c
\renewcommand{\c}[0]{{\mathbb C}}

\newcommand{\z}[0]{{\mathbb Z}}

\renewcommand{\r}[0]{{\mathbb R}}

\newcommand{\p}[0]{{\mathbb P}}

\newcommand{\q}[0]{{\mathbb Q}}
\newcommand{\map}[0]{\dasharrow}
\newcommand{\qtq}[1]{\quad\mbox{#1}\quad}

\newcommand{\rank}[0]{\operatorname{rank}}

\newcommand{\supp}[0]{\operatorname{Supp}}

\newcommand{\im}[0]{\operatorname{im}}

\newcommand{\sing}[0]{\operatorname{Sing}}

\newcommand{\onto}[0]{\twoheadrightarrow}

\newcommand{\tsum}[0]{\textstyle{\sum}}

\newcommand{\shom}[0]{\operatorname{\mathcal{H}\!\it{om}}}

\newcommand{\reg}[0]{\operatorname{\mathcal{R}}} 
\newcommand{\dist}[0]{\operatorname{dist}} 
\newcommand{\ctr}[0]{\operatorname{ctr}} 
\newcommand{\grass}[0]{\operatorname{Gr}}

\def\into{\DOTSB\lhook\joinrel\to}

\newcommand\vp{\varphi}
\newcommand\la{\lambda}
\newcommand\bR{\mathbb{R}}
\newcommand\cH{\mathcal{H}}
\newcommand\hensp[1]{\enspace\hbox{#1}\enspace} 
\newcommand\wth{\widetilde{H}}
\numberwithin{equation}{section}
\newcommand\glr{Glaeser refinement}
  \newcommand\varep{\varepsilon}
  \newcommand\tvp{\widetilde{\vp}}

\begin{document}
\bibliographystyle{amsalpha}


\title{Continuous linear combinations of polynomials}
\author{Charles Fefferman and J\'anos Koll\'ar}

\maketitle
\tableofcontents

\section{Introduction}
 
Let $f_1,\dots, f_r$ be polynomials or analytic functions on $\r^n$.
Our aim is to consider the following.

\begin{ques} \label{rq1}
 Which continuous functions  $\phi$ can be written in the form
\begin{equation}\label{1.1.eq}
\phi=\tsum_i \phi_if_i
\end{equation}
where the $\phi_i$ are  continuous functions?
Moreover, if $\phi$ has some regularity properties, can we chose the
$\phi_i$ to have the same (or some weaker) regularity properties?

\end{ques}

If the $f_i$ have no common zero, then a partition of unity argument
shows that every $\phi\in C^0(\r^n)$ can be written this way
and the $\phi_i$ have the same regularity properties 
(e.g., being H\"older, Lipschitz or  $C^m$) as
$\phi$.
By Cartan's Theorem B, if $\phi$ is real analytic then the
$\phi_i$ can also be chosen real analytic.

None of these hold if the common zero set $Z:=(f_1=\cdots=f_r=0)$
is not empty. Even if $\phi$ is a polynomial, the best one can say is that
the $\phi_i$ can be chosen to be   H\"older continuous;
see (\ref{counter.exmps}.1).
Thus the interesting aspects happen
 near the common zero set $Z$. 

The much studied $C^{\infty}$-version of Question \ref{rq1}
has a very different flavor  \cite{malgrange, tougeron}, but
the $L^{\infty}_{loc}$-version is  quite relevant:
Which  functions  can be written in the form
$\sum_i \psi_if_i$
where $\psi_i\in L^{\infty}_{loc}$?

The answer to the latter variant turns
out to be rather simple. If $\phi$ is such then 
$\phi/\sum_i |f_i|\in L^{\infty}_{loc}$.
Conversely, if this  holds then
$$
\phi=\sum_i \phi_if_i\qtq{where}\phi_i:= 
\frac{\phi}{\sum_j |f_j|}\cdot  \frac{\bar f_i}{|f_i|}
\in L^{\infty}_{loc}.
$$
Equivalently, the  obvious formulas
\begin{equation}\label{rq1.1}
\sum_i |f_i|=\sum_i  \frac{\bar f_i}{|f_i|}f_i
\qtq{and} \phi=\frac{\phi}{\sum_i |f_i|}\sum_i |f_i|
\end{equation}
show that 
$L^{\infty}_{loc}(\r^n)\cdot(f_1,\dots, f_r)$
is the  principal ideal generated by $\sum_i |f_i|$.
For many purposes it is even better to write $\phi$ as
\begin{equation}\label{rq1.2}
\phi=\sum_i \psi_if_i\qtq{where}\psi_i:= \frac{\phi\bar f_i}{\sum_j |f_j|^2} 
\in L^{\infty}_{loc}.
\end{equation}
Note that if $\phi$ is continuous (resp.\ differentiable) then the $\psi_i$
given in (\ref{rq1.2}) are  continuous  (resp.\ differentiable)
outside the common zero set $Z$; again indicating the special role of $Z$.

The above formulas also show that the
discontinuity of the $\psi_i$ along $Z$ can be removed
for certain functions.

\begin{lem}\label{zero.liit.lem}
 For a  continuous function $\phi$ the following are equivalent.
\begin{enumerate}
\item $\phi=\sum_i \phi_if_i$
where the $\phi_i$ are  continuous functions
such that $\lim_{x\to z} \phi_i=0$ for every $i$ and every $z\in Z$.
\item $\lim_{x\to z}  \frac{\phi}{\sum_i |f_i|}=0$ for  every $z\in Z$.\qed
\end{enumerate}
\end{lem}

Similar  conditions do not answer
Question \ref{rq1}.
First, if the $\psi_i$ defined in (\ref{rq1.2}) are  continuous,
then $\phi=\sum_i \psi_if_i$ is  continuous,
but frequently one can write $\phi=\sum_i \phi_if_i$
with $\phi_i$  continuous yet the formula
(\ref{rq1.2}) defines discontinuous functions $\psi_i$.
This happens already in very simple examples, like
$f_1=x, f_2=y$.  For $\phi=x$  (\ref{rq1.2}) gives
$$
x=\frac{x^2}{x^2+y^2}\cdot x+\frac{x y}{x^2+y^2}\cdot y
$$
whose coefficients are  discontinuous at the origin.

An even worse example is given by $f_1=x^2, f_2=y^2$ and $\phi=xy$.
Here $\phi$ can not be written as
$\phi=\phi_1f_1+\phi_2f_2$ but every inequality
 that is satisfied by $x^2$ and $y^2$ is also satisfied by
$\phi=xy$. We believe that there is no universal test or formula as
above  that answers Question \ref{rq1}.
At least it is clear that $C^0(\r^n)\cdot (x,y)$ is not a principal ideal
in $C^0(\r^n)$.

Nonetheless, these examples 
and the concept of axis closure defined by \cite{brenner}
suggest several simple necessary conditions.
These turn out to be  equivalent to each other, but
they do not settle Question \ref{rq1}.

 The algebraic version of Question \ref{rq1} was
posed by H.~Brenner, which led him to the notion
of the {\it continuous closure} of ideals  \cite{brenner}. 
We learned about it from a lecture of M.~Hochster.
It seems to us that the continuous version is the
more basic variant. In turn, the methods of  the
 continuous case can be used to settle several of the
algebraic problems \cite{k-hoch}.

\begin{say}[Pointwise tests]\label{rpwt}
For a continuous function $\phi$ and for a point  $p\in \r^n$
the following are equivalent.
\begin{enumerate}
\item For every sequence $\{x_j\}$ converging to $p$ there are
 $\psi_{ij}\in \c $ such that 
$\lim_{j\to \infty} \psi_{ij}$ 
exists for every $i$ and
$\phi(x_j)=\sum_i \psi_{ij}f_i(x_j)$  for every~$j$.
\item  We can write
$\phi=\sum_i \psi_i^{(p)}f_i$ where the $\psi_i^{(p)}(x)$ are continuous  at $p$.
\item  We can write
$\phi=\phi^{(p)}+\tsum_i c_i^{(p)} f_i$ where $c_i^{(p)}\in \c$ and
$\lim_{x\to p} \frac{\phi^{(p)}}{\sum_i |f_i|}=0$.
\end{enumerate}

 If $\phi=\tsum_i \phi_if_i$
where the $\phi_i$ are  continuous functions,
then we obtain the $\psi_{ij}, \psi_i^{(p)}$ 
by restriction and
$\phi=\bigl(\tsum_i (\phi_i-\phi_i(p))f_i\bigr)+\tsum_i \phi_i(p)f_i$
shows that $\phi$ satisfies the third test.
Conversely, if $\phi$  satisfies (3)  then
$\phi_p:=\phi-\sum_i c_i^{(p)} f_i$ is continuous and 
$\lim_{x\to p} \frac{\phi_p}{\sum_i |f_i|}=0$. 
By Lemma \ref{zero.liit.lem} we can write
$$
\phi=\sum_i \psi_i^{(p)}f_i\qtq{where}\psi_i^{(p)}:=
c_i^{(p)}+\frac{\phi_p\bar f_i}{ \sum_j |f_j|^2}
$$
and the $\psi_i^{(p)}(x)$ are continuous at $p$.
Thus (2) and (3)  are equivalent.
One can see their equivalence with (1) directly, but for us it is more
natural to obtain it by showing that they are all
  equivalent to the Finite set test to be introduced in
 (\ref{fin.set.test.defn}).

If the common zero set
$Z:=(f_1=\cdots=f_r=0)$ consists of a single point $p$, then
the   $\psi_i^{(p)}(x)$ constructed above 
are continuous everywhere. More generally,
if $Z$ is a finite set of points then 
these tests give  necessary and  sufficient conditions
for Question \ref{rq1}.
However, the following example of Hochster shows that
the  pointwise test for every $p$ does not give a  sufficient
condition in general.
\medskip

{\it \ref{rpwt}.4 Example.} \cite{hochster}
  Take $\{f_1, f_2, f_3\}:= \{x^2,y^2, xyz^2\}$ and $\phi:=xyz$.

Pick a point $p=(a,b,c)\in \r^3$.
If $c\neq 0$ then we can write
$$
xyz= \tfrac1{c}xyz^2+ \tfrac1{c}(c-z)xyz
\qtq{and}
\lim_{(x,y,z)\to (a,b,c)}\frac{(c-z)xyz}{|x^2|+|y^2|+| xyz^2|}=0,
$$
thus (\ref{rpwt}.3) holds.
Note that if $a=b=0$, then
$\frac1{c}xyz^2$  is the only possible constant coefficient
term that works. As $c\to 0$, the coefficient  $\frac1{c}$ is
not continuous, thus $xyz$ can not be written as
$xyz=\phi_1x^2+\phi_2y^2+\phi_3xyz^2$ where the $\phi_i$ are continuous.
Nonetheless,
 if $c=0$ then
$$
\lim_{(x,y,z)\to (a,b,0)}\frac{xyz}{|x^2|+|y^2|+| xyz^2|}=0.
$$
shows that (\ref{rpwt}.3) is satisfied (with all $c_i^{(a,b,0)}=0$).

One problem  is that the coefficients $c_i^{(p)}$ are
not continuous functions of $p$. In general, they are not even functions of
$p$ since a representation as in (\ref{rpwt}.2) or 
 (\ref{rpwt}.3) is not unique.
Still, this suggests a possibility of  reducing  Question \ref{rq1} to a similar
problem on the lower dimensional set $Z=(f_1=\cdots=f_r=0)$. 
\end{say}

We present two methods to answer   Question \ref{rq1}.

The first method starts with $f_1,\dots, f_r$ and $\phi$
and decides if $\phi=\sum_i \phi_if_i$ is solvable or not.
The union of the graphs of all discontinuous solutions
$(\phi_1,\dots, \phi_r)$ is a subset 
${\mathcal H}\subset \r^n\times \r^r$.
Then we use the tests  (\ref{rpwt}.1--3) repeatedly to get
smaller and smaller subsets of ${\mathcal H}$. 
After $2r+1$ steps, this process stabilizes.
This follows  \cite[Lem.2.2]{Fef06}. It  was
 adapted from a lemma in \cite{BMP03}, which in turn was adapted from a lemma
 in \cite{Gla58}. At the end we
 use Michael's theorem \cite{Mic56} to get a necessary and sufficient criterion.
This approach works even if the $f_i$ are continuous functions.
On the other hand,  its dependence on $\phi$ is somewhat delicate.

The second method uses in an essential way that the $f_i$ are
polynomials (or at least real analytic).
The method relies on the observation that formulas like
(\ref{rq1.1}--\ref{rq1.2}) give a continuous solution to
$\phi=\sum_i \phi_if_i$; albeit not on $\r^n$
but on some real algebraic variety  mapping to $\r^n$.
Following this idea, we transform the original
 Question \ref{rq1} on $\r^n$ to a similar problem
on a real algebraic variety $Y$ for which the
 solvability on any finite subset is equivalent to
continuous solvability.

The algebraic method also shows that if $\phi$ is
 H\"older continuous (resp.\ semialgebraic and  continuous) and  the 
 equation (\ref{1.1.eq}) has a
continuous solution then there is also a solution where
the  $\phi_i$ are  H\"older continuous (resp.\ semialgebraic and  continuous)
 (\ref{r.main.cors}).
By contrast, if can happen that $\phi$ is a
continuous rational function on $\r^3$,
the equation (\ref{1.1.eq}) has a continuous   semialgebraic solution
but has no continuous  rational solutions
\cite{k-rat-func}.

Both of the methods work for any linear system of equations
\begin{equation}\label{1.1.eq-syst}
\phi_j=\tsum_{i=1}^m \psi_if_{ij} \qtq{for $j=1,\dots, n$.}
\end{equation}

\section{The Glaeser--Michael method}

 Fix positive integers 
$n$, $r$ and let $Q$ be a compact metric space.

\begin{say}[Singular affine bundles]\label{thfi}
By a singular affine bundle (or {\it bundle} for short),
 we mean a family $\cH=(H_x)_{x\in Q}$ of 
affine subspaces $H_x\subseteq \bR^r$, parametrized by the points $x\in Q$. 
 The affine subspaces $H_x$ are the {\it fibers} of the bundle $\cH$.  
  (Here, we allow the empty set $\emptyset$ and the whole space $\bR^r$ as
 affine subspaces of $\bR^r$.)  A {\it section} of a bundle 
$\cH=(H_x)_{x\in Q}$ is a continuous map $f:Q\to\bR^r$ such that
 $f(x)\in H_x$ for each $x\in Q$.  We ask:
\begin{equation}\label{thfi.1}
\hbox{How can we tell whether a given bundle of $\cH$ has a section?}
\end{equation}
For instance, let $f_1,\dots,f_r$ and $\vp$ be given real-valued functions
 on $Q$.  For $x\in Q$, we take
\begin{equation}\label{thfi.2}
H_x=\{(\la_1,\dots,\la_r)\in\bR^r:\la_1f_1(x)+\dots+ \la_rf_r(x)=\vp(x)\}.
\end{equation}
Then a section $(\phi_1,\dots,\phi_r)$ of the bundle (\ref{thfi.2}) is
 precisely an $r$-tuple of continuous functions solving the equation
\begin{equation}\label{thfi.3}
\phi_1f_1+\dots+ \phi_rf_r=\vp\hensp{on} Q.
\end{equation}
To answer question (\ref{thfi.1}), we introduce the notion of
 ``Glaeser refinement''.  (Compare with \cite{Gla58},
 \cite{BMP03}, \cite{Fef06}.)  Let $\cH=(H_x)_{x\in Q}$ be a bundle. 
 Then the {\it Glaeser refinement} of $\cH$ is the bundle
 $\cH'=(H'_x)_{x\in Q}$, where, for each $x\in Q$, 
\begin{equation}\label{thfi.4}
H'_x=\{\la\in H_x:\dist (\la,H_y)\to 0\hensp{as}
 y\to x\quad (y\in Q)\}.
\end{equation}
One checks easily that
\begin{equation}\label{thfi.5}
\cH'\hbox{ is a {\it subbundle} of $\cH$, i.e., $H'_x\subseteq H_x$ 
for each $x\in Q$}
\end{equation}
and
\begin{equation}\label{thfi.6}
\hbox{the bundles $\cH$ and $\cH'$ have the same sections.}
\end{equation}

Starting from a given bundle $\cH$, and iterating the above construction,
 we obtain a sequence of bundles $\cH^0,\cH^1,\cH^2,\cdots$, where
 $\cH^0=\cH$, and $\cH^{i+1}$ is the Glaeser refinement of $\cH^i$ for
 each $i$.  In particular, $\cH^{i+1}$ is a subbundle of $\cH^i$, and
 all the bundles $\cH^i$ have the same sections.
\end{say}

We will prove the following results.

\begin{lem}[Stabilization Lemma]\label{35.7}
$\cH^{2r+1}=\cH^{2r+2}=\cdots$
\end{lem}

\begin{lem}[Existence of Sections]\label{35.8} Let $\cH=(H_x)_{x\in Q}$ be a 
bundle.  Suppose that $\cH$ is its own Glaeser refinement, and suppose each 
fiber $H_x$ is non-empty.  Then $\cH$ has a section.
\end{lem}

The above results allow us to answer question (\ref{thfi.1}).  Let $\cH$ 
be a bundle, let $\cH^0,\cH^1,\cH^2,\cdots$ be its iterated Glaeser
 refinements, and let $\cH^{2r+1}=(\widetilde H_x)_{x\in Q}$.  Then 
$\cH$ has a section if and only if each fiber $\wth_x$ is non-empty.

The bundle (\ref{thfi.2}) provides an interesting example.  One checks 
that its Glaeser refinement is given by $\cH^1=(H^1_x)_{x\in Q}$, where 
$$
H^1_x=\Big\{(\la_1,\dots,\la_r)\in \bR^r:
\big| \tsum^r_1\la_if_i(y)-\vp(y)\big| =o
\big(\tsum^r_1|f_i(y)|\big) \hensp{as} y \to x\Big\}.
$$
Thus, the necessary condition (\ref{rpwt})  for the existence of continuous 
solutions of (\ref{thfi.3})  asserts precisely that the fibers 
$H^1_x$ are all non-empty.

In Hochster's example (\ref{rpwt}.4),
 the equation (\ref{thfi.3}) has no continuous solutions, because the 
second Glaeser refinement $\cH^2 = (H^2_x)_{x\in Q}$ has an empty fiber, 
namely $H^2_0$.

We present self-contained proofs of (\ref{35.7}) and (\ref{35.8}), for the
 reader's 
convenience.  A terse discussion would simply note that the proof of
 \cite[Lem.2.2]{Fef06} also yields (\ref{35.7}), and that one can easily prove 
(\ref{35.8})
 using Michael's theorem \cite{Mic56}, \cite{BL00}.

\begin{say}[Proof of the Stabilization Lemma]\label{psl}
Let $\cH^0,\cH^1,\cH^2,\cdots$ be the iterated \glr s of $\cH$; and 
let $\cH^i=(H^i_x)_{x\in Q}$ for each $i$.

We must show that $H^\ell_x=H^{2r+1}_x$ for all $x\in Q$, $\ell\geq 2r+1$. 
 If $H^{2r+1}_x=\emptyset$, then the desired result is obvious.  

For non-empty $H^{2r+1}_x$, it follows at once from the following.
\medskip

{\it Claim \ref{psl}.1$_k$. Let $x\in Q$.  If $\dim H^{2k+1}_x\geq r-k$, 
then $H^\ell_x=H^{2k+1}_x$ for all $\ell \geq 2k+1$.}
\medskip

We prove (\ref{psl}.1$_k$) for all $k\geq 0$, by induction on $k$. 
 In the case $k=0$, (\ref{psl}.1$_k$) asserts that
\begin{equation}\label{psl.2}
\hbox{If $H^1_x=\bR^r$, then $H^\ell_x=\bR^r$ for all $\ell\geq 1$.}
\end{equation}
By definition of \glr, we have
\begin{equation}\label{psl.3}
\dim H^{\ell+1}_x\leq \liminf_{y\to x}\dim H^\ell_y.
\end{equation}
Hence, if $H^1_x=\bR^r$, then $H^0_y=\bR^r$ for all $y$ in a neighborhood of $x$. Consequently, $H^\ell_y=\bR^r$ for all $y$ in a neighborhood of $x$, and for all $\ell \geq 0$.  This proves (\ref{psl}.1$_k$) in the base case $k=0$.
For the induction step, we fix $k$ and assume (\ref{psl}.1$_k$)
 for all $x\in Q$.  We will prove (\ref{psl}.1$_{k+1}$).  We must show that
\begin{equation}\label{psl.4}
\hbox{If $\dim H^{2k+3}_{x} \geq r-k-1$, then $H^\ell_x=H^{2k+3}_x$ for all
 $\ell\geq 2k+3$.}
\end{equation}
If $\dim H^{2k+1}_x\geq r-k$, then (\ref{psl.4}) follows at once from 
(\ref{psl}.1$_k$).  Hence, in proving (\ref{psl.4}),
 we may assume that $\dim H^{2k+1}_x\leq r-k-1$.  Thus,
\begin{equation}\label{psl.5}
\dim H^{2k+1}_x=\dim H^{2k+2}_x = \dim H^{2k+3}_x = r-k-1.
\end{equation}
We now show that
\begin{equation}\label{psl.6}
H^{2k+2}_y = H^{2k+1}_1 \hensp{for all $y$ near enough to} x.
\end{equation}
If fact, suppose that (\ref{psl.6}) fails, i.e., suppose that
\begin{equation}\label{psl.7}
\dim H^{2k+2}_y \leq \dim H^{2k+1}_y-1\hensp{for y arbitrarily close to} x.
\end{equation}
For $y$ as in (\ref{psl.7}), our inductive assumption (\ref{psl}.1$_k$)
 shows that $\dim H^{2k+1}_y \allowbreak\leq r-k-1$.  Therefore, for $y$ 
arbitrarily near $x$, we have
\[
\dim H^{2k+2}_y \leq \dim H^{2k+1}_y-1\leq r-k-2.\]
Another application  of (\ref{psl.3}) now yields $\dim H^{2k+3}_x\leq r-k-2$, 
contradicting (\ref{psl.5}).  Thus, (\ref{psl.6}) cannot fail.

From (\ref{psl.6}) we see easily that $H^\ell_y=H^{2k+3}_y$ for all $y$
 near enough to $x$, and for all $\ell\geq 2k+3$.

This completes the inductive step (\ref{psl.4}), and proves the
 Stabilization Lemma.\qed
\end{say}

\begin{say}[Proof of Existence of Sections]\label{pes}
We give the standard proof of Michael's theorem in the relevant special case. 
 We start with a few definitions.  If $H\subset \bR^r$ is an affine subspace
 and $v\in \bR^r$ is a vector, then $H-v$ denotes the translate
 $\{w-v:w\in H\}$.  If $\cH=(H_x)_{x\in Q}$ is a bundle, and if $f:Q\to \bR^r$
 is a continuous map, then $\cH-f$ denotes the bundle $(H_x-f(x))_{x\in Q}$. 
 Note that if $\cH$ is its own \glr\ and has non-empty fibers, then the same
 is true of $\cH-f$.

Let $\cH=(H_x)_{x\in Q}$ be any bundle with non-empty fibers.  We define the 
norm $\|\cH\|: = \sup_{x\in Q}\dist(0,H_x)$.  Thus, $\|\cH\|$ is a
 non-negative real number or $+\infty$.

Now suppose that $\cH=(H_x)_{x\in Q}$ is a bundle with non-empty fibers, and
 suppose that $\cH$ is its own \glr.

\begin{prop} $\|\cH\|<+\infty$.\end{prop}

\begin{proof} Given $x\in Q$, we can pick $w_x\in H_x$ since $H_x$ is non-empty.  Also, $\dist (w_x,H_y)\to 0$ as $y\to x$ $(y\in Q)$, since $\cH$ is its own \glr.  Hence, there exists an open ball $B_x$ centered at $x$, such that $\dist(w_x,H_y)\leq 1$ for all $y\in Q\cap B_x$.  If follows that $\dist (0,H_y)\leq |w_x|+1$ for all $y\in Q\cap B_x$.  We can cover the compact space $Q$ by finitely many of the open balls $B_x$ $(x\in Q)$; say, 
\[
Q\subset B_{x_1}\cup B_{x_2}\cup\dots\cup B_{x_N}.\]
Since $\dist(0,H_y)\leq |w_{x_i}|+1$ for all $y\in Q\cap B_{x_i}$, it follows that
\[
\dist(0,H_y)\leq \max\{|w_{x_i}|+1:i=1,2,\dots,N\}\hensp{for all}y\in Q.\]
Thus $\|\cH\|<+\infty$.\end{proof}

\begin{prop} \label{thni}
Given $\varepsilon>0$, there exists a continuous map $g:Q\to \bR^r$ such that
$$
\label{thni.on}
\dist(g(y),H_y)\leq \varepsilon\hensp{for all}y\in Q,$$
and
$$
\label{thni.tw}
|g(y)|\leq \|\cH\|+\varepsilon\hensp{for all}y\in Q.$$\end{prop}

\begin{proof}
Given $x\in Q$, we can find $w_x\in H_x$ such that $|w_x|\leq \|\cH\|+\varepsilon$.  We know that $\dist(w_x,H_y)\to 0$ as $y\to x$ ($y\in Q$), 
since $\cH$ is its own \glr.  Hence, there exists an open ball $B(x,2r_x)$ 
centered at $x$, such that
$$\label{thni.th}
\dist (w_x,H_y)<\varep\hensp{for all}y\in Q\cap B(x,2r_x).$$
The  compact space $Q$ may be covered by finitely many of the open balls $B(x,r_x)$ ($x\in Q$); say
$$\label{thni.fo}
Q\subset B(x_1,r_{x_1})\cup\dots\cup B(x_N,r_{x_N}).$$
For each $i=1,\dots,N$, we introduce a non-negative continuous function $\widetilde{\vp}_i$ on $\bR^n$, supported in $B(x_i,2r_{x_i})$ and equal to one on $B(x_i,r_{x_i})$.  We then define
$\vp_i(x)=\tvp_i(x)/ (\tvp_1(x)+\dots+\tvp_N(x))$ for $i=1,\dots,N$ and $x\in Q$.
(This makes sense, thanks for \eqref{thni.fo}.)

The $\vp_i$ form a partition of unity on $Q$:
\begin{itemize}
\item Each $\vp_i$ is a non-negative continuous function on $Q$, equal to zero outside $Q\cap B(x_i,2r_{x_i})$; and
\item $\sum^N_{i=1}\vp_i=1$ on  $Q$.
\end{itemize}
We define
\[
g(y)=\sum^N_{i=1}w_{x_i}\vp_i(y)\hensp{for} y\in Q.\]
Thus, $g$ is a continuous map from $Q$ into $\bR^r$.  Moreover, \eqref{thni.th} shows that  $\dist(w_{x_i},H_y)\leq \varep$ whenever $\vp_i(y)\ne 0$.  Therefore,
\begin{align*}
\dist (g(y),H_y)&\leq \sum^N_{i=1}\dist(w_{x_i}, H_y)\vp_i(y)\\
&\leq \varep\sum^N_{i=1}\vp_i(y)=\varep\hensp{for all} y\in Q.
\end{align*}
Also, for each $y\in Q$ we have
\[
|g(y)|\leq \sum^N_{i=1}|w_{x_i}| \vp_i(y)\leq \sum^N_{i=1}(\|\cH\|+\varep)\vp_i(y)=\|\cH\|+\varep.\]
  The proof of Proposition \ref{thni} is complete.
\end{proof}

\begin{cor} \label{foze}Let $\cH$ be a bundle with non-empty fibers, equal to its own \glr.  Then there exists a continuous map $g:Q\to\bR^r$, such that
$
\|\cH-g\| \leq \frac12\|\cH\|$, and 
$|g(y)| \leq 2\|\cH\|$ for all $y\in Q$. \end{cor}

\begin{proof} If $\|\cH\|>0$, then we can just take $\varep=\frac12\|\cH\|$ in Proposition \ref{thni}.  If instead $\|\cH\|=0$, then we can just take $g=0$.\end{proof}

Now we can prove the existence of sections.  Let $\cH=(H_x)_{x\in Q}$ be a bundle.  Suppose the $H_x$ are all non-empty, and assume that $\cH$ is its own \glr.  By induction on $i=0,1,2,\dots$, we define continuous maps $f_i,g_i:Q\to \bR^r$.  We start with $f_0=g_0=0$.  Given $f_i$ and $g_i$, we apply Corollary \ref{foze} to the bundle $\cH-f_i$, to produce a continuous map $g_{i+1}:Q\to \bR^r$, such that
$\|(\cH-f_i)-g_{i+1}\|\leq \frac12 \|\cH-f_i\|$, and $|g_{i+1}(y)|\leq 2\|\cH-f_i\|$ for all $y\in Q$.

We then define $f_{i+1}=f_i+g_{i+1}$.  This completes our inductive definition of the $f_i$ and $g_i$.  Note that $f_0=0$, $\|\cH-f_{i+1}\|\leq \frac12 \|\cH-f_i\|$ for each $i$, and $|f_{i+1}(y)-f_i(y)|\leq 2\|\cH-f_i\|$ for each $y\in Q$, $i\geq 0$.  Therefore, $\|\cH-f_i\|\leq 2^{-i}\|\cH\|$ for each $i$, and $|f_{i+1}(y) - f_i(y)|\leq 2^{1-i} \|\cH\|$ for each $y\in Q$, $i\geq 0$.
In particular, the $f_i$ converge uniformly on $Q$ to a continuous map $f:Q\to \bR^r$, and 
$\|\cH-f_i\|\to 0$ as $i\to \infty$. 

Now, for any $y\in Q$, we have 
\begin{align*}
\dist(f(y),H_y)&=\lim_{i\to\infty}\dist(f_i(y),H_y)\\
&=\lim_{i\to\infty}\dist\bigl(0,H_y-f_i(y)\bigr)\leq \liminf_{i\to\infty}\|\cH-f_i\|=0.
\end{align*}
Thus, $f(y)\in H_y$ for each $y\in Q$.  Since also $f:Q\to \bR^r$ is a continuous map, we see that $f$ is a section of $\cH$.  This completes the proof of existence of sections.\hfill\qed
\end{say}

\begin{say}[Further problems and remarks]\label{foon.on}
We return to the equation
\begin{equation}\label{foon.on.1}
\phi_1f_1+\dots+\phi_rf_r=\vp\hensp{on}\bR^n,
\end{equation}
where $f_1,\dots,f_r$ are given polynomials.

Let $X$ be a function space, such as $C^m_{\rm loc}(\bR^n)$ or
 $C^\alpha_{\rm loc}(\bR^n)$ $(0<\alpha\leq 1)$.  It would be interesting to 
know how to decide whether the equation (\ref{foon.on.1}) admits a solution
 $\phi_1,\dots,\phi_r\in X$.  
 Some related examples are given in (\ref{counter.exmps}).
  If $\vp$ is  real-analytic, and if (\ref{foon.on.1})
admits a continuous solution, then we can take the continuous functions 
$\phi_i$ to be real-analytic outside the common zeros of the $f_i$. 
 To see this we invoke the following

\begin{thm}[Approximation Theorem, see \cite{narasimhan}]
 Let $\phi,\sigma:\Omega\to \bR$ be
 continuous functions on an open set $\Omega\subset \bR^n$, and 
suppose $\sigma>0$ on $\Omega$.  Then there exists a real-analytic 
function $\tilde \phi:\Omega\to\bR$ such that 
$|\tilde \phi(x)-\phi(x)|\leq \sigma(x)$
 for all $x\in  \Omega$.
\end{thm}

Once we know the Approximation Theorem, we can easily correct a continuous
 solution $\phi_1,\dots,\phi_r$ of (\ref{foon.on.1}) so that the  functions
 $\phi_i$ are real-analytic outside the common zeros of $f_1,\dots,f_r$. 
 We take
$\Omega=\{x\in \bR^n:f_i(x)\ne 0\hensp{for some} i\}$, and set 
$\sigma(x)=\sum_i (f_i(x))^2$ for $x\in \Omega$.

We obtain real-analytic functions $\tilde \phi_i$ on $\Omega$ such that
 $|\tilde \phi_i-\phi_i|\leq\sigma$ on $\Omega$.  Setting 
$h=\sum_i \tilde \phi_if_i-\vp= \sum_i(\tilde \phi_i-\phi_i)f_i$ on $\Omega$
 and then defining
\[
\left\{\begin{array}{l}
\phi^{\#}_i=\tilde \phi_i-\frac{h f_i}{f^2_1+\dots+f^2_r}\hensp{on}\Omega\\
\phi^{\#}_i=\phi_i\hensp{on} \bR^n\setminus \Omega\end{array}\right\},\]
we see that $\sum_i \phi^{\#}_i f_i=\vp$, with 
$\phi^{\#}_i$ continuous on $\bR^n$
 and real-analytic on $\Omega$.
\end{say}

\section{Computation of the solutions}\label{sec.3}

In this section, we show how to compute a continuous solution 
$( \phi_1 , \ldots , \phi_r ) $ of the equation  
\begin{equation}\label{1}
 \phi_1 f_1 + \ldots + \phi_r f_r = \phi ,
\end{equation}
 assuming such a solution exists. 
We start with an example, then spend several sections explaining how to compute Glaeser refinements and sections of bundles, and finally return to (1) in the general case.  

For our example, we pick Hochster's equation  
\begin{equation}\label{2}
\phi_1 x^2 + \phi_2 \ y^2 + \phi_3 \ xyz^2 = \phi \quad  \text{on} \ \ Q = [-1, 1]^{3} , 
\end{equation}
where $\phi$ is a given, continuous, real-valued function on $Q$. Our goal here is to compute a continuous solution of (\ref{2}),
 assuming such a solution exists. 

Suppose $\phi_1 , \phi_2 , \phi_3$ satisfy (\ref{2}). Then, for every positive integer $\nu$, we have  
$$   \phi_1 \big( \tfrac{1}{\nu} , 0, z \big)   \cdot \tfrac{1}{\nu^2} = \phi \big( \tfrac{1}{\nu} , 0, z \big) , $$ 
$$   \phi_2 \big( 0 , \  \tfrac{1}{\nu} , \ z \big)   \cdot \tfrac{1}{\nu^2} = \phi \big( 0 , \  \tfrac{1}{\nu} , \ z \big) , \ \text{and} $$ 
$$   \phi_1 \big( \tfrac{1}{\nu} , \tfrac{1}{\nu} , z \big)   \cdot \tfrac{1}{\nu^2} +   \phi_2 \big( \tfrac{1}{\nu} , \tfrac{1}{\nu} , z \big)   \cdot \ \tfrac{1}{\nu^2} +   \phi_3 \big( \tfrac{1}{\nu} , \  \tfrac{1}{\nu} , \ z \big)   \cdot \ \tfrac{z^2}{\nu^2} = \phi \big( \tfrac{1}{\nu} , \tfrac{1}{\nu} , z \big) $$
for all $z \in [-1, 1].$ Hence, it is natural to define  
\begin{eqnarray}
 \xi_{1}(z)& =& \lim_{\nu \to \infty} \nu^2 \cdot \phi \big( \tfrac{1}{\nu} , 0, z \big),\label{3} \\
\xi_{2}(z) &=& \lim_{\nu \to \infty} \nu^2 \cdot \phi \big( 0, \tfrac{1}{\nu} , z \big) \qtq{and}\label{4}\\
\xi_{3}(z) &=& \lim_{\nu \to \infty} \nu^2 \cdot \phi \big( \tfrac{1}{\nu} , \tfrac{1}{\nu} , z \big)  \qtq{ for $z \in [-1, 1]$.} \label{5}
\end{eqnarray}

If~(\ref{2}) has a continuous solution $\overrightarrow{\phi} = (\phi_1 , \phi_2 , \phi_3)$, then the limits~(\ref{3}) exist, and our solution $\overrightarrow{\phi}$ satisfies  
\begin{eqnarray}
 \phi_{1}(0, 0, z) = \xi_{1}(z),\quad
 \phi_{2}(0, 0, z) = \xi_{2}(z),
\qtq{and} \label{6}\\
\phi_{1}(0, 0, z) + \phi_2 (0, 0, z) + z^2 \phi_3(0, 0, z) = \xi_3(z)\label{7}
\end{eqnarray}
 for $z \in [-1, 1]$, so that
\begin{equation}\label{8}
 \phi_{3}(0, 0, z) = z^{-2} \ \cdot \  [ \xi_{3}(z) - \xi_{1}(z) - \xi_2(z) ] \quad \text{for} \ z \  \in [-1, 1] \smallsetminus \{ 0 \}.
\end{equation}
\indent To recover $\phi_{3}(0,0,0),$ we just pass to the limit in~(\ref{8}). Let us define
\begin{equation}\label{9}
 \xi = \lim_{\nu \to \infty} \nu^2 \cdot   \xi_{3} \big( \tfrac{1}{\nu} \big) - \xi_{1} \big( \tfrac{1}{\nu} \big) - \xi_{2} \big( \tfrac{1}{\nu} \big)   .
\end{equation}
If~(\ref{2}) has a continuous solution $\overrightarrow{\phi}$, then the limit~(\ref{9}) exists, and we have 
\begin{equation}\label{10}
\phi_{3}(0, 0, 0) = \xi.
\end{equation}
Thus, $\overrightarrow{\phi}(0, 0, z)(z \ \in \ [-1, 1])$ can be computed from the given function $\phi$. Note that $\phi_{3}(0, 0, 0)$ arises from $\phi$ by taking an iterated limit. 

Since we assumed that $\overrightarrow{\phi}$ is continuous, we have in particular
\begin{equation}\label{11}
\text{The functions } \ \ \phi_{i}(0, 0, z) \ \ (i=1, 2, 3) \text{ are continuous on } [-1, 1].
\end{equation}
\noindent From now on, we regard  
$\overrightarrow{\phi}(0, 0, z) = (\phi_{1}(0, 0, z), \phi_{2}(0, 0, z), \phi_{3}(0, 0, z))$  
as known. 

Let us now define
\begin{equation}\label{12}
\overrightarrow{\phi}^{\#}(x,y,z)  =  \overrightarrow{\phi}(x,y,z) - \overrightarrow{\phi}(0, 0, z)=
 (\phi^{\#}_{1}(x,y,z), \phi^{\#}_{2}(x,y,z), \phi^{\#}_{3}(x,y,z))
\end{equation}
and
\begin{equation}\label{13} \phi^{\#}(x,y,z) =  \phi (x,y,z) - [\phi_{1}(0,0,z) \cdot x^{2} + \phi_{2}(0, 0, z) \cdot y^{2} + \phi_{3}(0,0,z) \cdot xyz^{2}]  
\end{equation}
on $Q$.  
Then, since $\overrightarrow{\phi}$ is a continuous solution of~(\ref{2}), we see that  
\begin{equation}\label{14}
 \phi^{\#} \text{ and all the }  \phi^{\#}_{i} \text{ are continuous functions on } Q ; 
 \end{equation}
\begin{equation}\label{15}
 \phi^{\#}_{i}(0, 0, z) = 0  \qquad \text{ for all } z \ \in [-1, 1], \ \ \ i =1, 2, 3; \text{ and}  
 \end{equation}
\begin{equation}\label{16}
 \phi^{\#}_{1}(x,y,z) \cdot x^2 + \phi^{\#}_{2}(x,y,z) \cdot y^2 + \phi^{\#}_{3}(x,y,z) \cdot xyz^2 = \phi^{\#} (x,y,z) \text{ on } Q.
 \end{equation}
 
We don't know the functions $\phi^{\#}_{i}(i=1,2,3)$ , but $\phi^\#$ may be computed from the given function $\phi$ in~(\ref{2}), since we have already computed $\phi_{i}(0,0,z)  
(i=1,2,3). $  (See~(\ref{13}).)  

We now define $\overrightarrow{\Phi}^{\#} (x,y,z) = (\Phi^{\#}_{1} (x,y,z), \Phi^{\#}_{2} (x,y,z), \Phi^{\#}_{3} (x,y,z))$ to be the shortest vector $(v_1, v_2, v_3) \ \in \mathbb{R}^3$ such that
\begin{equation}\label{17}
 v_1 \cdot x^2 + v_2 \cdot y^2 + v_3 \cdot xyz^2 = \phi^{\#}(x,y,z). 
\end{equation}
\indent Thus,  
\begin{equation}\label{18}
 \Phi^{\#}_{1}(x,y,z) \cdot x^2 + \Phi^{\#}_{2}(x,y,z) \cdot y^2 + \Phi^{\#}_{3}(x,y,z) \cdot xyz^2 = \phi^{\#} (x,y,z) \text{ on } Q.
 \end{equation}
\indent Unless $x=y=0$, we have
\begin{equation}\label{19}
\begin{array}{lll}
 \Phi^{\#}_{1}(x,y,z) &=& \frac{x^2}{x^4 + y^4 + x^2y^2 z^4} \cdot \phi^{\#}(x,y,z),\\
 \Phi^{\#}_{2}(x,y,z)  &=& \frac{y^2}{x^4 + y^4 + x^2y^2 z^4} \cdot \phi^{\#}(x,y,z),\\
\Phi^{\#}_{3}(x,y,z) & =&\frac{xyz^2}{x^4 + y^4 + x^2y^2 z^4} \cdot \phi^{\#}(x,y,z)
\end{array}
 \end{equation}
\begin{equation}\label{22}
\text{ If } x = y = 0 , \ \ \text{ then }  \Phi^{\#}_{i}(x,y,z) = 0 \ \ \ \text{ for } \  i = 1,2,3.
 \end{equation} \hspace{.1in} 
\noindent Since $\phi^\#$ may be computed from $\phi$, the functions $\Phi^{\#}_{i} \ \ \ (i=1,2,3)$ may also be computed from $\phi$. 

Recall that $\overrightarrow{\phi}^{\#} = (\phi^{\#}_{1} ,  \phi^{\#}_{2} , \phi^{\#}_{3})$ satisfies~(\ref{16}). Since $\overrightarrow{\Phi} (x,y,z)$ was defined as the shortest vector satisfying~(\ref{17}), we learn that
\begin{equation}\label{23}
 \bigm| \ \overrightarrow{\Phi}^{\#}(x,y,z) \ \bigm| \ \leq  \ \bigm| \overrightarrow{\phi}^{\#}(x,y,z) \bigm| \ \ \text{ for all } (x , y  , z) \ \in \ Q .
 \end{equation}
Since also $\overrightarrow{\phi}^{\#}$ satisfies~(\ref{14}) and~(\ref{15}), it follows that 
\begin{equation}\label{24}
\Phi^{\#}_{i}(x,y,z) \ \to 0 \text{ as } (x, y,z) \to (0, 0, z'), \text{ for each } i=1,2,3 . 
 \end{equation}
Here, $z' \in [-1, 1]$ is arbitrary.  

We will now check that
\begin{equation}\label{25}
\Phi^{\#}_{1} ,  \Phi^{\#}_{2} , \Phi^{\#}_{3} \text{ are continuous functions on } Q .
 \end{equation}
Indeed, the  $\Phi^{\#}_{i}$ are continuous at each $(x,y,z) \ \in \ Q$ such that $(x,y) \neq (0,0)$, as we see at once from~(\ref{14}) and~(\ref{19}) $\cdots$~(\ref{21}). On the other hand,~(\ref{22}) and~(\ref{24}) tell us that the $\Phi^{\#}_{i}$ are continuous at each $(x,y,z) \ \in \ Q$ such that $(x,y) = (0, 0).$ Thus,~(\ref{25}) holds.  

Next, we set
\begin{equation}\label{26}
 \Phi_{i} (x,y,z) = \Phi^{\#}_{i}(x,y,z) \ + \ \phi_{i}(0, 0, z) \ \text{ for } (x,y,z) \ \in \ Q , \ \ \ i = 1,2,3.
 \end{equation}
\indent Since $\Phi^{\#}_{i}(x,y,z)$ and $\phi_{i}(0, 0, z)$ can be computed from $\phi$, the same is true of $\Phi_{i} (x, y, z)$.  

Also,~(\ref{11}) and~(\ref{25}) imply
\begin{equation}\label{27}
 \Phi_{1} , \Phi_{2} , \Phi_3 \text{ are continuous functions on } Q .
 \end{equation}
From~(\ref{13}),~(\ref{18}) and~(\ref{26}), we have
\begin{equation}\label{28}
 \Phi_{1} (x,y,z) \cdot x^2 +  \Phi_{2} (x, y, z) \cdot y^2 + \Phi_3 (x,y,z) \cdot xyz^2 = \phi(x,y,z) \text{ on } Q .
 \end{equation}
 
Note also that the $\Phi_{i}$ satisfy the estimate
\begin{equation}\label{29}
\max_{ \qquad \qquad  {x \epsilon Q}, \ \ {i=1,2,3}} \bigm| \Phi_{i} (x) \bigm| \ \leq \ C \hspace*{-.4in} \max_{ {\qquad \qquad x \epsilon Q}, \ \ {i=1,2,3}} \bigm| \phi_{i}(x) \bigm|
 \end{equation}
for an absolute constant $C$, as follows from~(\ref{13}),~(\ref{23}) and~(\ref{26}).  

Let us summarize the above discussion of equation~(\ref{2}). Given a function $\phi : Q \to \mathbb{R}$, we proceed as follows.  
\medskip

\noindent {\it Step 1}: We compute the limits~(\ref{3}),~(\ref{4}),~(\ref{5}) for each $z \ \in[-1, 1]$, to obtain the functions $\xi_{i}(z) \quad (i=1,2,3)$.  

\noindent {\it Step 2}: We compute the limit~(\ref{9}), to obtain the number $\xi$. 

\noindent {\it Step 3}: We read off the functions $\phi_{i}(0,0,z) \qquad (i=1,2,3)$ from~(\ref{6}),~(\ref{7}),~(\ref{8}),~(\ref{10}). 

\noindent {\it Step 4}: We compute the function $\phi^{\#}(x,y,z)$ from~(\ref{13}). 

\noindent {\it Step 5}: We compute the functions $\Phi^{\#}_{i}(x,y,z) \qquad (i=1,2,3) $ from 
\noindent ~(\ref{19}) $\cdots$~(\ref{22}). 

\noindent {\it Step 6}: We read off the functions $\Phi_{i}(x,y,z) \qquad (i=1,2,3)$ from~(\ref{26}).  
\medskip

 If, for our given $\phi$, equation~(\ref{2}) has a continuous solution $(\phi_1, \phi_2, \phi_3)$, then the limits exist in Steps 1 and 2, and the above procedure produces continuous functions $\Phi_1, \Phi_2, \Phi_3$ that solve equation~(\ref{2}) and satisfy estimate~(\ref{29}). 

\noindent If instead the equation~(\ref{2}) has no continuous solutions, then we cannot guarantee that the limits in Steps 1 and 2 exist. It may happen that those limits exist, but the functions $\Phi_1, \Phi_2, \Phi_3$ produced by our procedure are discontinuous.  

This concludes our discussion of example~(\ref{2}). We devote the next several sections to making calculations with bundles. We show how to pass from a given bundle to its iterated Glaeser refinements by means of formulas involving iterated limits. After recalling the construction of ``Whitney cubes'' (which will be used below), we then provide additional formulas to compute a section of a given Glaeser stable bundle with non-empty fibers. These results together allow us to compute a section of any given bundle for which a section exists. Finally, we apply our results on bundles, to provide a discussion of equation~(\ref{1}) in the general case, analogous to the discussion given above for example~(\ref{2}).  
 
\subsection{Computation of the Glaeser refinement}{\ }
\label{subsection.3.1}

We use the standard inner product on $\mathbb{R}^r$. We define a 
{\it homogeneous bundle} to be a family 
 $\cH^0 = (H^0_x)_{x \in Q}$ of vector subspaces $H^0_x \subset  \mathbb{R}^r$, indexed by the points $x$ of a closed cube $Q  \subset \mathbb{R}^n$. We allow $\{0\}$ and $\mathbb{R}^r$, but not the empty set, as vector subspaces of $\mathbb{R}^r$. Note that the fibers of a homogeneous bundle are vector subspaces of $\mathbb{R}^r$, while the fibers of a bundle are (possibly empty) affine subspaces of $\mathbb{R}^r$.  

Any bundle $\cH$ with non-empty fibers may be written uniquely in the form
\begin{equation}\label{1A}
\cH = (H_x)_{x \in Q} = (v(x) + H^0_x)_{x \in Q},
\end{equation}
where $
\cH^0 = (H_x^0)_{x \in Q}$ is a homogeneous bundle, and
$v(x) \perp \ H^0_x $ for each $ x \in Q.$
 
\indent Let $\widetilde{\cH}$ be the Glaeser refinement of $\cH$, and suppose $\widetilde{\cH}$ has non-empty fibers. Just as $\cH$ may be written in the form~(\ref{1A}), we can express $\widetilde{\cH}$ uniquely in the form
\begin{equation}\label{4A}
\widetilde{\cH} = (\widetilde{v}(x) + \widetilde{H}^0_x)_{x \in Q} ,
\end{equation}
where  $\widetilde{\cH}^0 = (\widetilde{H}^0_x)_{x \in Q}$
 is a homogeneous bundle, and 
$\widetilde{v}(x) \perp \widetilde{H}^0_x \text{ for each } x \in Q.$

One checks easily that $\widetilde{\cH}^0$ is the Glaeser refinement of $\cH^0$. 
The goal of this section is to understand how the vectors $\widetilde{v} (x) (x \in Q)$ depend on the vectors $v(y) (y \in Q)$ for fixed $\cH^0$. 

To do so, we introduce the sets
\begin{eqnarray}
E &=& \{(x, \lambda) \in Q \times \mathbb{R}^r : \lambda \perp H^0_x \}, \text{ and }\label{7A}\\
\Lambda(x) &=& \{ \widetilde{\lambda} \in \mathbb{R}^r : (x, \widetilde{\lambda} ) \text{ belongs to the closure of } E \} \text{ for } x \in Q.\label{8A}.
\end{eqnarray}

The following is immediate from the definitions~(\ref{7A}),~(\ref{8A}). 

\begin{claim}\label{9A}
Given $ \widetilde{\lambda} \in \Lambda (x),$ there exist points 
$ y^\nu \in Q $ and vectors $ \lambda^\nu \in \mathbb{R}^r (\nu  \geqslant 1),$
  such that  
$ y^\nu \to x $ and $ \lambda^\nu \to \widetilde{\lambda} \text{ as } \nu \to \infty ,$ and $ \lambda^\nu \perp H^0_{y^\nu}$ for each $ \nu.$\qed
\end{claim}
 
 Note that $E$ and $\Lambda (x)$ depend on $\cH^0$, but not on the vectors $v (y), \ y \in Q$. The basic properties of $\Lambda (x)$ are given by the
 following result. 

\begin{lem}\label{p12.lem} Let $x \in Q$. Then
\begin{enumerate}
\item \label{10A}
Each  $\widetilde{\lambda} \in \Lambda (x)$ is perpendicular to $ \widetilde{H}^0_x.$
\item \label{11A}
Given any vector  $ \widetilde{v} \in \mathbb{R}^r $ not belonging to 
$ \widetilde{H}^0_x, $  there exists a vector $ \lambda \in \Lambda(x)$ such that $ \lambda \cdot \widetilde{v} \neq 0.$
\item \label{12A}
The vector space $ (\widetilde{H}^0_x)^\perp \subset \mathbb{R}^r$
 has a basis $ \widetilde{\lambda}_1 (x),\ \ldots ,\   
\widetilde{\lambda}_s (x)  $ consisting entirely of vectors $ \widetilde{\lambda}_i (x) \in \Lambda (x).$
\end{enumerate}
 \end{lem}

{\it Proof}: To check~(\ref{10A}), let $\widetilde{\lambda} \in \Lambda (x)$ and let $\widetilde{v} \in \widetilde{H}^0_x$. We must show that $\widetilde{\lambda} \cdot \widetilde{v} = 0$. Let $y^{\nu} \in Q $ and $\lambda^{\nu} \in \mathbb{R}^r \ ( \nu \geqslant 1)$ be as in~(\ref{9}). Since $\widetilde{v} \in \widetilde{H}^0_x$ and $(\widetilde{H}^0_y)_{y \in Q}$ is the Glaeser refinement of $(H^0_y)_{y \in Q}$, we know that distance $( \widetilde{v}, H^0_y) \to 0$ as $y \to x$. In particular, distance $( \widetilde{v}, H^0_{y^{\nu}}) \to 0$ as $\nu \to \infty$. Hence, there exist $v^{\nu} \in H^0_{y^{\nu}} \ \ (v \geqslant 1)$ such that $v^{\nu} \to \widetilde{v}$ as $\nu \to \infty$. 
Since $v^{\nu} \in H^0_{y^{\nu}}$ and $\lambda^{\nu} \perp H^0_{y^{\nu}}$, we have $\lambda^{\nu} \cdot v^{\nu} = 0$ for each $\nu$. Since $\lambda^{\nu} \to \widetilde{\lambda}$ and $v^{\nu} \to \widetilde{v}$ as $\nu \to \infty$, it follows that $\widetilde{\lambda} \cdot \widetilde{v} = 0$, proving~(\ref{10A}). 

To check~(\ref{11A}), suppose $\widetilde{v} \in \mathbb{R}^r$ does not belong to  $\widetilde{H}^0_x$. Since $(\widetilde{H}^0_y)_{y \in Q}$ is the Glaeser refinement of $(H^0_y)_{y \in Q}$, we know that distance $(\widetilde{v}, H^0_y)$ does not tend to zero as $y \in Q$ tends to $x$. Hence there exist $\epsilon > 0$ and a sequence of points $y^{\nu} \in Q \ \ (\nu \geqslant 1)$, such that  
\begin{equation}\label{13A}
y^\nu \to x \text{ as } \nu \to \infty, \qtq{ but }
\dist (\widetilde{v} , H^0_{y^\nu}) \geqslant \epsilon \text{ for each } \nu.
\end{equation}
Thanks to~(\ref{14}), there exist unit vectors $\lambda^\nu \in \mathbb{R}^r \ \ (\nu \geqslant 1)$, 
such that
\begin{equation}\label{15A}
\lambda^\nu \perp H^0_{y^\nu}
\qtq{and} \lambda^\nu \cdot \widetilde{v} \geqslant \epsilon
\qtq{for each $\nu$. }
\end{equation}
Passing to a subsequence, we may assume that the vectors $\lambda^\nu$ tend to a limit $\widetilde{\lambda} \in \mathbb{R}^r $as $\nu \to \infty$. 

Comparing~(\ref{15A}) to~(\ref{7A}), we see that $(y^\nu, \lambda^\nu) \in E $ for each $\nu$. Since $y^\nu \to x$ and $\lambda^\nu \to \widetilde{\lambda}$ as $\nu \to \infty$, the point $(x, \widetilde{\lambda})$ belongs to the closure of $E$, hence $\widetilde{\lambda} \in \Lambda (x)$. 
Also, $\widetilde{\lambda} \cdot \widetilde{v} = \lim_{\nu \to \infty}\limits \lambda^\nu \cdot \widetilde{v} \geqslant \epsilon$ by~(\ref{16}); in particular, $\widetilde{\lambda} \cdot \widetilde{v} \neq 0$. The proof of~(\ref{11A}) is complete. 
Finally, to check~(\ref{12A}), we note that  
$$
\bigcap_{\widetilde{\lambda} \in \Lambda (x)}\limits (\widetilde{\lambda}^\perp) = \widetilde{H}^0_x, \quad \text{thanks to~(\ref{10}) and~(\ref{11}) .}
$$
Assertion~(\ref{12A}) now follows from linear algebra. 
The proof of Lemma \ref{p12.lem} is complete. \qed  

Let $\widetilde{\lambda}_1 (x), \cdots , \widetilde{\lambda}_s (x)$ be the basis for $(\widetilde{H}^0_x)^\perp$ given by~(\ref{12A}), and let 
$  \widetilde{\lambda}_{s+1} (x), \cdots , \widetilde{\lambda}_r (x) $
be a basis for $\widetilde{H}^0_x.$  Thus
\begin{equation}\label{18A}
 \widetilde{\lambda}_1 (x), \cdots , \widetilde{\lambda_r} (x) \  
\text{ form a basis for } \mathbb{R}^r.
\end{equation} 
For $1 \leq i \leq s$, the vector $\widetilde{\lambda}_i (x)$ belongs to $\Lambda(x)$. Hence, by~(\ref{9A}), there exist vectors $\lambda^\nu_i (x) \in \mathbb{R}^r$ and points $y^\nu_i (x) \in Q \ \ (\nu \geqslant 1)$, such that  
\begin{eqnarray}
y^\nu_i(x) \to x \ \text{ as } \nu \to \infty ,\label{19A}\\
\lambda^\nu_i (x) \to \widetilde{\lambda}_i (x) \text{ as } \nu \to \infty, \text{ and }\label{20A}\\
\lambda^\nu_i (x) \perp H^0_{y{^\nu_i} (x)} \text{ for each } \nu.\label{21A}
\end{eqnarray}
\noindent For $s + 1 \leq i \leq r$, we take $y^\nu_i (x) = x$ and $\lambda^\nu_i (x) = 0 \ \ (\nu \geqslant 1)$. 
Thus,~(\ref{19}),~(\ref{21}) hold also for $s + 1 \leq i \leq r$, although~(\ref{20A}) holds only for $1 \leqslant i \leqslant s$. 

We now return to the problem of computing $\widetilde{v} (x) ( x \in Q)$ for the bundles given by~(\ref{1A}) and~(\ref{4A}). The answer is as follows. 

\begin{lem}\label{22A} Given $x \in Q$, we have
$\widetilde{\lambda}_i (x) \cdot \widetilde{v} (x) = \lim_{\nu \to \infty}\limits \ \ \lambda^\nu_i (x) \cdot v (y^\nu_i (x))$ for $ i=1, \cdots , r.$ 
In particular, the limit in~(\ref{22A}) exists. 
\end{lem}

{\it Remarks}: Since $\widetilde{\lambda}_1 (x), \cdots , \widetilde{\lambda}_r (x)$ form a basis for $\mathbb{R}^r$,~(\ref{22A}) completely specifies the vector $\widetilde{v} (x)$. Note that the points $y^\nu_i (x)$ and the vectors $\widetilde{\lambda}_i (x), \ \lambda^\nu_i (x)$ depend only on $\cH^0$, not on the vectors $v(y) \ \ (y \in Q).$ 

 {\it Proof}: 
First suppose that $1 \leq i \leq s$. Since $\widetilde{v} (x) $ belongs to the fiber $\widetilde{v} (x) + \widetilde{H}^0_x$ of the Glaeser refinement of $(v(y) + H^0_y)_{y \in Q}$, we know that  
 $\dist(\widetilde{v} (x), v(y) + H^0_y) \to 0$ as $y \to x \ (y \in Q)$. 
In particular,  
 $\dist(\widetilde{v} (x), v(y^\nu_i (x)) + H^0_{{y^\nu_i} (x)}) \to 0 \ $ as $ \ \nu \to \infty$. 
Hence, there exist vectors $w^\nu_i(x) \in H^0_{{y^\nu_i} (x)}$ such that 
$v(y^\nu_i (x)) + w^\nu_i (x) \to \widetilde{v} (x) \ $ as $\ \nu \to \infty$. 
Since also  
$\lambda^\nu_i (x)  \to \widetilde{\lambda}_i (x)$ as $\nu \to \infty$,
 it follows that  
$\widetilde{\lambda}_i (x) \cdot \widetilde{v} (x) = \lim_{v \to \infty}\limits \ \lambda_i^\nu (x) \cdot [v(y^\nu_i (x)) + w^\nu_i (x)].$ 
 However, since $w^\nu_i (x) \in H^0_{{y^\nu_i} (x)}$ and $\lambda_i^\nu (x) \perp H^0_{{y^\nu_i} (x)}$, 
we have $\lambda_i^\nu (x) \cdot w^\nu_i (x) = 0$ for each $\nu$. 

Therefore,  
$\widetilde{\lambda}_i^\nu (x) \cdot \widetilde{v} (x) = \lim_{\nu \to \infty}\limits \ \lambda_i^\nu (x) \cdot v (y^\nu_i (x))$, i.e.~(\ref{22}) holds for $1 \leq i \leq s$. 

On the other hand, suppose $s+1 \leq i \leq r$. Then since  
$\widetilde{\lambda}_i(x) \in \widetilde{H}^0_x$ and $\widetilde{v} (x) \perp \widetilde{H}^0_x$, we have 
$\widetilde{\lambda}_i (x) \ \cdot \ \widetilde{v} (x) = 0$. Also, in this case we defined $\lambda^\nu_i (x) = 0$, 
hence $\lambda^\nu_i (x) \cdot v (y^\nu_i (x)) = 0$ for each $\nu$. Therefore,  
$\widetilde{\lambda}_i(x) \cdot \widetilde{v} (x) = 0 = \lim_{\nu \to \infty}\limits \lambda^\nu_i (x) \cdot v (y^\nu_i (x))$, so that~(\ref{22A}) holds also for
$s+1 \leq i \leq r$. The proof of Lemma \ref{22A} is complete. \qed 

\subsection{Computation of iterated Glaeser refinements}{\ }
\label{subsection.3.2}

In this section, we apply the results of the preceding section to study iterated Glaeser refinements. 
Let  
$\cH = (v(x) + H^0_x)_{x \in Q}$
 be a bundle, given in the form (\ref{1A}). 
We assume that $\cH$ has a section. Therefore, $\cH$ and all its iterated Glaeser refinements have non-empty fibers. For $\ell \geq 0$, we write the $\ell^{th}$ iterated Glaeser refinement in the form  
\begin{equation}\label{2B}
\cH^{( \ell )} = (v^\ell (x) + H^{0, \ell}_x)_{x \in Q}, 
\end{equation}
 where 
$\cH^{0, \ell} = (H^{0, \ell}_x)_{x \in Q}$ is a homogeneous bundle, and
$v^\ell (x) \perp H^{0, \ell}_x \ \ \text{ for each } x \in Q.$
(Again, we use the standard inner product on $\mathbb{R}^r$.) 
In particular, $\cH^{( 0 )} = \cH$, and  
\begin{equation}\label{5B}
\cH^{0,0} = (H^0_x)_{x \in Q}, \text{ with } H^0_x \text{ as in }~(\ref{1}).
\end{equation}
 One checks easily that $\cH^{0, \ell}$ is the $\ell ^{th}$ iterated Glaeser refinement of $\cH^{0, 0}$. Our goal here is to give formulas computing $v^\ell (x)$ in terms of the $v(y) (y \in Q)$ in~(\ref{1}). 

We proceed by induction on $\ell$. For $\ell = 0$, we have  
\begin{equation}\label{6B}
v^0(x) = v (x)  \text{ for all } x \in Q.
\end{equation}
For $\ell \geq 1$, we apply the results of the preceding section, to pass from $(v^{\ell - 1}(x))_{x \in Q}\ \ $  to $\ \ (v^\ell (x))_{x \in Q}$.

\begin{claim}\label{7B--12B} We obtain points   
$
y_i^{\ell , \nu} (x) \in Q \qquad (\nu \geq 1, \ \ 1 \leq i \leq r, \ \ x \in Q);
$
and vectors  
$\widetilde{\lambda}^\ell_i (x) \in \mathbb{R}^r \quad (1 \leq i \leq r, \ \ x \in Q), \widetilde{\lambda}^{\ell, \nu}_i (x) \ \ \ (1 \leq i \leq r, \quad \nu \geq 1, \quad x \in Q)
$
with the following properties. 
\begin{enumerate}
\item  
 The above points and vectors depend only on $ \cH^{0,0},$  
 not on the family of vectors $ (v(x))_{x \in Q}$,
\item 
$\widetilde{\lambda}^\ell_1 (x), \cdots , \widetilde{\lambda}^\ell_r (x)$
 form a basis of $ \mathbb{R}^r,$ for each $ \ell \geq 1, x \in Q.$
\item 
$y_i^{\ell , \nu} (x) \to x $ as $ \nu  \to \infty $ for each $ \ell \geq 1, \ \ 1 \leq i \leq r, \ \ x \in Q.$
\item \label{12B}
$
[\widetilde{\lambda}^\ell_i(x) \cdot v^\ell(x)] = \lim_{\nu \to \infty}\limits [\widetilde{\lambda}^{\ell , \nu}_i (x) \cdot v^{\ell - 1} (y^{\ell , \nu}_i (x))]$
 for each $ \ell \geq 1,\  1 \leq i \leq r, \  x \in Q.$
\end{enumerate}
\end{claim}

The last formula
 computes the $v^\ell(x) \ \  (x \in Q)$ in terms of the $v^{\ell - 1}(y) \ (y \in Q)$ for $\ell \geq 1$, completing our induction on $\ell$. 

Note that we have defined the basis vectors $\widetilde{\lambda}^\ell_1(x), \cdots , \widetilde{\lambda}^\ell_r (x)$ only for $\ell \geq 1$. For $\ell = 0,$ it is convenient to use the standard basis vectors for $\mathbb{R}^r$ , i.e., we define 
\begin{equation}\label{13B}
\begin{split}
\widetilde{\lambda}^0_i (x) = (0, 0, \cdots, 0, 1, 0, \cdots, 0) \in \mathbb{R}^r,  
\text{with the 1 in the } i^{th} \text{ slot.}
\end{split}
\end{equation}
It is convenient also to set  
\begin{equation}\label{14B}
\xi^\ell_i (x) = \widetilde{\lambda}^\ell_i (x) \cdot v^\ell (x) \text{ for } x \in Q , \ \ell \geq 0,\quad 1 \leq i \leq r,
\end{equation}
\noindent and to expand $\widetilde{\lambda}^{\ell , \nu}_i (x) \in \mathbb{R}^r$ in terms of the basis $\widetilde{\lambda}^{\ell - 1}_1 (y), \cdots , \widetilde{\lambda}^{\ell - 1}_r (y) \quad  \text{ for } y = y^{\ell , \nu}_i (x)$. Thus, for suitable coefficients 
$\beta^{\ell , \nu}_{i j} (x) \in \mathbb{R} \ \ (\ell \geq 1, \ \nu \geq 1 , \ \ 1 \leq i \leq r , \ \ 1 \leq j \leq r, \ \ x \in Q)$ 
we have 
\begin{equation}\label{15B}
\widetilde{\lambda}^{\ell , \nu}_i (x) = \sum^r\limits_{ij}\limits \beta^{\ell , \nu}_{ij} (x) \cdot \widetilde{\lambda}^{\ell - 1}_j \big( y^{\ell , \nu}_i (x) \big) \ \text{ for }  x \in Q, \ \ell \geq 1 , \nu \geq 1, \ \  1 \leq i \leq r.
\end{equation}
Note that  
the coefficients $ \beta^{\ell , \nu}_{i j} (x)$ depend only on $ \cH^{0, 0},  $
not on the vectors $ v(y) (y \in Q)$.
 
Putting~(\ref{14B}) and~(\ref{15B}) into~(\ref{7B--12B}.\ref{12B}), 
we obtain a recurrence relation for the $\xi^\ell_i (x)$:
\begin{equation}\label{17B}
\xi^\ell_i (x) = \lim_{\nu \to \infty}\limits \ \sum^r\limits_{j=1}\limits \beta^{\ell , \nu}_{i j} (x) \ \cdot \ \xi^{\ell - 1}_j \big( y^{\ell , \nu}_i (x) \big)
\qtq{for} \ell \geq 1, \ 1 \leq i \leq r, \ x \in Q.
\end{equation}
 For $\ell = 0$,~(\ref{6B}),~(\ref{13B}) and~(\ref{14B}) give 
\begin{equation}\label{18B}
\noindent \xi^0_i (x) = [ i^{th} \text{ component of } v(x)].
\end{equation}
Since $\beta^{\ell , \nu}_{i j} (x) \text{ and } y^{\ell , \nu}_i (x)$ are independent of the vectors $v(y) (y \in Q)$, our formulas~(\ref{17B}),~(\ref{18}) express each $\xi^{\ell}_{i} (x)$ as an iterated limit in terms of the vectors $v(y) (y \in Q)$. In particular, the $\xi^{\ell}_{i} (x)$ depend linearly on the $v(y) \ (y \in Q)$. 

We are particularly interested in the case $\ell = 2r + 1$, since the bundle $\cH^{2r+1}$ is Glaeser stable, as we proved in section X. 

Since $\widetilde{\lambda}^{2r+1}_1 (x), \cdots , \widetilde{\lambda}^{2r+1}_r (x)$ form a basis of $\mathbb{R}^r$ for each $x \in Q$, there exist vectors 
 $w_1 (x), \cdots, w_r (x) \in \mathbb{R}^r  $ for each $x \in Q$, 
such that 
\begin{equation}\label{19B}
v = \sum^r_{i=1}   \widetilde{\lambda}_i^{2r+1}(x) \cdot v   w_i (x) \text{ for any vector } v \in \mathbb{R}^r, \text{ and for any } x \in Q.
\end{equation}
Note that  
the vectors $ w_1(x), \cdots , w_r (x) \in \mathbb{R}^r $ depend only on 
$\cH^{0,0},  $ not on the vectors $ v(y) (y \in Q)$.
 
Taking $v = v^{2r+1} (x)$ in~(\ref{19B}), and recalling~(\ref{14B}), we see that 
\begin{equation}\label{21B}
\noindent v^{2r+1} (x) = \sum^r\limits_{i=1}\limits \xi^{2r+1}_i (x) w_i (x) \qquad \text{ for each } x \in Q. 
\end{equation}
Thus, we determine the $\xi^\ell_i (x)$ by the recursion~(\ref{17B}),~(\ref{18B}), and then compute $v^{2r+1} (x)$ from formula~(\ref{21B}). Since also $(H^{0,2r + 1}_x)_{x \in Q}$ is simply the $(2r + 1)^{rst}$ 

 Glaeser refinement of $\cH^{0,0}$, we have succeeded in computing the Glaeser stable bundle 
$(v^{2r+1} (x) + H^{0, 2r + 1}_x)_{x \in Q}$ 
 in terms of the initial bundle as in~(\ref{1A}). 

Our next task is to give a formula for a section of a Glaeser stable bundle. To carry this out, we will use ``Whitney cubes'', a standard construction which we explain below. 

\subsection{Whitney cubes}{\ }
 \label{subsection.3.3}

In this section, for the reader's convenience, we review ``Whitney cubes'' 
(see \cite{malgrange, St, whi}). 
We will work with closed cubes $Q \subset \mathbb{R}^n$ whose sides are parallel to the coordinate axes. We write $\ctr(x)$ and $\delta_Q$ to denote the center and side length of $Q$, respectively; and we write $Q^*$ to denote the cube with center $\ctr(Q)$ and side length $3 \delta$. 

To ``bisect'' $Q$ is to write it as a union of $2^n$ subcubes, each with side length $\frac{1}{2} \delta_Q$, in the obvious way; we call those $2^n$ subcubes the ``children'' of $Q$. 

Fix a cube $Q^o$. The ``dyadic cubes'' are the cube $Q^o$, the children of $Q^o$, the children of the children of $Q^o$, and so forth. Each dyadic $Q$ is a subcube of $Q^o$. If $Q$ is a dyadic cube other than $Q^o$, then $Q$ is a child of one and only one dyadic cube, which we call $Q^+$. Note that $Q^+ \subset Q^*$. 

Now let $E_1$ be a non-empty closed subset of $Q^o$. A dyadic cube $Q \neq Q^o$ will be called a ``Whitney cube'' if it satisfies
\begin{eqnarray}
\dist (Q^*, E_1) \geq \delta_Q, \text{ and}\label{1C}\\
\dist ((Q^+)^*, E_1) < \delta_{Q^+}.\label{2C}
\end{eqnarray}
The next result gives a few basic properties of Whitney cubes. In this section, we write $c, C, C^\prime$, etc. to denote constants depending only on the dimension $n$. These symbols need not denote the same constant in different occurrences. 

\begin{lem}\label{p16.lem} For each Whitney cube $Q$, we have
\begin{enumerate}
\item \label{3C}
$\delta_Q \leq \dist (Q^*, E_1) \leq C \delta_Q;$  in particular,
\item \label{4C}
$Q^* \cap E_1 = \phi.$
\item \label{5C}
The union of all Whitney cubes is $ Q^o \ \smallsetminus \ E_1.$
\item \label{6C}
Any given $ y \in Q^o \smallsetminus E_1 $ has a neighborhood that meets $ Q^*$
 for at most $ C $ distinct Whitney cubes $ Q$.
\end{enumerate}
\end{lem}

{\it Proof}: Estimates~(\ref{3C}) follow at once from (1) and (2); and (4) is immediate from (3). 

To check~(\ref{5C}), we note first that each Whitney cube $Q$ is contained in $Q^o \smallsetminus E_1$, thanks to~(\ref{4C}) and our earlier remark that every dyadic cube is contained in $Q^o$. Conversely, let $x \in Q^o \smallsetminus E_1$ be given. Any small enough dyadic cube $\widehat{Q}$ containing $x$ will satisfy~(\ref{1C}). Fix such a $\widehat{Q}$. There are only finitely many dyadic cubes $Q$ containing $x$ with side length greater than or equal to $\delta_{\widehat{Q}}$. Hence, there exists a dyadic cube $Q \ni x$ satisfying~(\ref{1C}), whose side length is at least as large as that of any other dyadic cube $Q^\prime \ni x$ satisfying~(\ref{1C}) . We know that $Q \neq Q^o$, since~(\ref{1C}) fails for $Q^o$. Hence, $Q$ has a dyadic parent $Q^+$. We know that~(\ref{1C}) fails for $Q^+$, since the side length of $Q^+$ is greater than that of $Q$. It follows that $Q$ satisfies~(\ref{2C}). Thus $Q \ni x$ is a Whitney cube, completing the proof of~(\ref{5C}). 

We turn our attention to~(\ref{6C}). Let $y \in Q^o \smallsetminus E_1$. We set $r=10^{-3}$ distance $(y, E_1)$, and we prove that there are at most $C$ distinct Whitney cubes $Q$ for which $Q^*$ meets the ball $B(x,r)$. 

Indeed, let $Q$ be such a Whitney cube. Then there exists $z \in B (y, r) \cap Q^*$. By~(\ref{3D}), we have
\begin{equation}\label{7C}
\delta_Q \leq \dist (z, E_1) \leq C \delta_Q.  
\end{equation}
Since $z \in B(y,r)$, we know that 
$\lvert \dist (z, E_1) - \dist (y, E_1) \ \ \rvert \leq 10^{-3} \dist (y, E_1)$. 
 Hence
\begin{equation}\label{8C}
(1 - 10^{-3}) \dist (y, E_1) \leq \dist (z, E_1) \leq (1 + 10^{-3}) \dist (y, E_1).
\end{equation}
From~(\ref{7C}),~(\ref{8C}) we learn that 
\begin{equation}\label{9C}
c \dist (y, E_1) \leq \delta_Q \leq C \dist (y, E_1).
\end{equation}
Since $z \in B(y,r) \cap Q^*$, we know also that 
\begin{equation}\label{10C}
\dist (y,Q^*) \leq \dist (y, E_1).
\end{equation}
\noindent For fixed $y$, there are at most $C$ distinct dyadic cubes that satisfy~(\ref{9C}),~(\ref{10C}). 

Thus,~(\ref{6}) holds and Lemma \ref{p16.lem} is proven. \qed 
\medskip

The next result provides a partition of unity adapted to the geometry of the Whitney cubes. 

\begin{lem}\label{p17.lem}
 There exists a collection of real-valued functions $\theta_Q$ on $Q^o$, indexed by the Whitney cubes $Q$, satisfying the following conditions.
\begin{enumerate}
\item \label{11C}
 Each $ \theta_Q $ is a non-negative continuous function on $ Q^o.$
\item \label{12C}
For each Whitney cube $ Q,$ the function $ \theta_Q $ is zero on $ Q^o \smallsetminus Q^*.$
\item \label{13C}
$ \textstyle{\sum_Q} \theta_Q = 1$  on $ Q^o \smallsetminus E_1.$
\end{enumerate}
\end{lem}
 
{\it Proof}: Let $\widetilde{\theta}(x)$ be a non-negative, continuous function on $\mathbb{R}^n$, such that 
 $\widetilde{\theta}(x) = 1 \quad$ for $x = (x_{1}, \cdots , x_n)$ with $\max \ \{ \lvert x_1 \rvert, \cdots , \lvert x_n \rvert \} \leq \frac{1}{2}$ 
 and  
$\widetilde{\theta}(x) = 0 \quad \text{ for } x = (x_1, \cdots , x_n) \text{ with } \max \ \{ \lvert x_1 \ \rvert, \cdots , \lvert x_n \rvert \} \geq 1$. 

 For each Whitney cube $Q$, define  
 $\widetilde{\theta}_Q(x) = \widetilde{\theta} \big( \frac{x - \ctr(Q)}{\delta_Q} \big), \text{ for } x \in \mathbb{R}^n$. 
 Thus, $\widetilde{\theta}_Q$ is a non-negative continuous function on $\mathbb{R}^n$, equal to $1 \text{ on } Q$, and equal to $0 \text{ outside } Q^*$. 
\noindent It follows easily, thanks to~(\ref{5C}) and~(\ref{6C}), that $\sum_{Q^\prime}\limits \widetilde{\theta}_{Q^\prime}$ is a non-negative continuous function on $Q^o \smallsetminus E_1$, greater than or equal to one at every point of $Q^o \smallsetminus E_1$. 

Consequently, the functions $\theta_Q$, defined by  
 $\theta_Q (x) = \widetilde{\theta}_Q (x) \diagup   \sum_{Q^\prime}\limits \ \ \widetilde{\theta}_{Q^\prime} (x)  $ for $x \in Q^o \smallsetminus E_1$, 
 $\theta_Q (x) = 0 \ \  \text{ for } x \in E_1$, 
 are easily seen to satisfy~(\ref{11C}),~(\ref{12C}),~(\ref{13C}). \qed
\medskip

Additional basic properties of Whitney cubes, and sharper versions of Lemma 
\ref{p17.lem} may be found in \cite{malgrange, St, whi}. 

The partition of unity $\{\theta_Q\}$ on $Q^o \smallsetminus E_1$ is called the ``Whitney partition of unity''. 
 
\subsection{The Glaeser--stable case}
 \label{subsection.3.4}

 In this section, we suppose we are given a Glaeser-stable bundle with 
non-empty fibers, written in the form 
\begin{equation}\label{1D}
 \cH = (v(x) + H^0_x)_{x \in Q}, 
\end{equation}
 where 
$\cH^0 = (H^0_x)_{x \in Q} $  is a homogeneous bundle, and 
\begin{equation}\label{3D}
v(x) \bot H^0_x \quad \text{ for each } x \in Q.
\end{equation}
 (As before, we use the standard inner product on $\mathbb{R}^r$.) 
 Our goal here is to give a formula for a section $F$ of the bundle $\cH$. 
We will take  
\begin{eqnarray}
F(x) = \sum_{y \in S(x)}\limits A(x,y) v(y) \in \mathbb{R}^r \text{ for each } x \in Q, \text{ where}\label{4D}\\
 S(x) \subset Q \text{ is a finite set for each } x \in Q,
\text{ and}\label{5D}\\
A(x,y): \mathbb{R}^r \to \mathbb{R}^r \text{ is a linear map, for each } x \in Q, y \in S(x).\label{6D}
\end{eqnarray}
 Here  
the sets $ S(x)$ and the linear maps $ A(x,y)$ are determined by $ \cH^0; $
  they do not depend on the family of vectors $ (v(x))_{x \in Q}.$
 
We will establish the following result. 

\begin{thm}\label{p18.thm}
 We can pick the $S(x)$ and $A(x,y)$ so that~(\ref{5D}),~(\ref{6D}) hold, and the function $F: Q \to \mathbb{R}^r$, defined by~(\ref{4D}), is a section of the bundle $\cH$. Moreover, that section satisfies 
\begin{enumerate}
\item \label{8D}
$\max_{\quad x \in Q} \ \lvert F(x) \ \rvert \leq C \sup_{ x \in Q} 
\lvert v(x) \rvert , $ where $ C $  depends only on $ n $ and $ r.$
\item \label{9D}
Furthermore, each of the sets $ S(x)$ contains at most $ d $ points,  
where $ d $ depends only on $ n $ and $ r.$
\end{enumerate}
\end{thm}
 
\noindent {\it Note}: Since $v(x)$ is the shortest vector in $v(x) + H^0_x$ by~(\ref{3D}), it follows that  
$\sup_{\ x \in Q} \ \lvert v(x) \rvert \ = \sup_{\ x \in Q}$ distance $(0, v(x) + H^0_x) = \parallel \cH \parallel < \infty$; see our earlier discussion  of Michael's Theorem. 

\noindent {\it Proof}: Roughly speaking, the idea of our proof is as follows. We partition $Q$ into finitely many ``strata'', among which we single out the ``lowest stratum'' $E_1$. For $x \in E_1$, we simply set $F(x) = v(x)$. To define $F$ on $Q \smallsetminus E_1$, we cover $Q \smallsetminus E_1$ by Whitney cubes $Q_\nu. \text{ Each } Q{^*_{\nu}} \text{ fails to meet } E_1$, by definition, and therefore has fewer strata than $Q$. Hence, by induction on the number of strata, we can produce a formula for a section $F_{\nu}$ of the bundle $\cH$ restricted to $Q{^*_\nu}$. Patching together the $F_{\nu}$ by using the Whitney partition of unity, we define our section $F$ on $Q \smallsetminus E_1$, and complete the proof of Theorem \ref{p18.thm}. 

Let us begin our proof. For $k=0, 1, \cdots, r,$ the $k^{th}$ ``stratum'' of $\cH$ is defined by 
\begin{equation}\label{10D}
E(k) = \{x \in Q: \dim \ H^0_x = k \}.
\end{equation}
\noindent The ``number of strata'' of $\cH$ is defined as the number of non-empty $E(k)$; this number is at least 1 and at most $r+1$. 
\noindent We write $E_1$ to denote the stratum $E(k_{\min})$, where $k_{\min}$ is the least $k$ such that $E(k)$ is non-empty. We call $E_1$ the ``lowest stratum''. 

We will prove Theorem  \ref{p18.thm} by induction on the number of strata, allowing the constants $C \text{ and } d$ on~(\ref{8D}),~(\ref{9D}), to depend on the  number of strata, as well as on $n \text{ and } r$.  Since the number of strata is at most $r+1$, such an induction will yield Theorem  \ref{p18.thm} as stated. 

Thus, we fix a positive integer $\Lambda$, and assume the inductive hypothesis: 
\begin{enumerate}
\item[(H1)] Theorem  \ref{p18.thm} holds, with constants 
$ C_{\Lambda - 1}, d_{\Lambda - 1} $  in~(\ref{8}),~(\ref{9}),
whenever the number of strata is less than $ \Lambda.$
\end{enumerate}

We will then prove Theorem  \ref{p18.thm}, with constants $C_\Lambda, d_\Lambda$ in~(\ref{8D}),~(\ref{9D}), whenever the number of strata is equal to $\Lambda$. Here, $C_\Lambda$ and $d_\Lambda$ are determined by $C_{\Lambda -1}, d_{\Lambda -1}, n \text{ and } r$. To do so, we start with~(\ref{1D}),~(\ref{3}), and assume that  
\begin{enumerate}
\item[(H2)] The number of strata of $ \cH $ is equal to $ \Lambda.$
\end{enumerate}
We must produce sets $S(x)$ and linear maps $A(x,y)$ satisfying~(\ref{5D}) $\cdots$~(\ref{9D}), with constants $C_\Lambda, d_\Lambda$ depending only on $C_{\Lambda-1}, d_{\Lambda -1}, n, r. $ This will complete our induction, and establish Theorem  \ref{p18.thm}. 

 For the rest of the proof of Theorem  \ref{p18.thm}, we write $c, C, C^\prime$, etc. to denote constants determined by $C_{\Lambda - 1}, d_{\Lambda -1}, n, r$. These symbols need not denote the same constant in different occurrences. 

The following useful remark is a simple consequence of our assumption that the bundle~(\ref{1D}) is Glaeser stable. Let $ x \in E(k), $ and let
\begin{equation}\label{13D}
 v_1, \cdots , v_{k+1} \ \in \ \ v (x) + H^0_x 
\end{equation}
 be the  vertices of a non-degenerate affine $ k$-simplex in 
$ \mathbb{R}^r. $
Given $ \epsilon > 0 \text{ there exists } \delta > 0 $ such that for any 
$ y \in Q \cap B(x, \delta), $
there exist $ v^\prime_1 , \cdots , v^\prime_{k+1} \in v(y) + H^0_y $ satisfying 
$ \lvert v^\prime_i - v_i  \rvert \ < \ \epsilon$ for each $ i$. 
Here, as usual, $B(x, \delta)$ denotes the ball of radius $\delta$ about $x$. 

Taking $\epsilon$ small enough in~(\ref{13D}), we conclude that $v^\prime_1, \cdots , v^\prime_{k+1} \in v(y) + H^0_y$ are the vertices of a non-degenerate affine $k$-simplex in $\mathbb{R}^r$. Therefore,~(\ref{13D}) yields at once 
that if $ x \in E(k)$, then $ \ \dim
 \ H^0_y \geq k $ for all 
$ y \in Q $
sufficiently close to $ x$.
 In particular, 
the lowest stratum $ E_1 $ is a non-empty closed subset of $ Q$.
  Also, for each $k=0, 1, 2, \cdots , r$,~(\ref{13D}) shows that
the map
\begin{equation}\label{16D}
 x \mapsto v(x) + H^0_x 
\end{equation}
 is continuous from $ E(k)$ to the space of  
 all affine $ k$-dimensional subspaces of $ \mathbb{R}^r$.

Since each $H^0_x$  is a vector subspace of $\mathbb{R}^r$,
 we learn from~(\ref{3D}) and~(\ref{16D}) that 
the map $ x \mapsto v(x) $ is continuous on each $ E(k)$.
  In particular, 
\begin{equation}\label{18D}
 x \mapsto v(x) \text{ is continuous on } E_1.
\end{equation}
 
  Next, we introduce the Whitney cubes $\{Q_\nu \}$ and the Whitney partition of unity $\{\theta_\nu\}$ for the closed set $E_1 \subset Q$. From the previous section, we have the following results. We write $\delta_\nu$ for the side length of the Whitney cube $Q_\nu$. 
Note that
\begin{eqnarray}
\delta_\nu \leq \dist (Q^*_\nu, E_1) \leq C \delta_\nu \text{ for each } \nu.
\label{19D}\\
Q^*_\nu \cap E_1 = \phi \quad \text{ for each } \nu.\label{20D}\\
\bigcup_\nu\limits Q_\nu = Q \smallsetminus E_1.\label{21D}\\
\begin{array}{c}
\text{ Any given } y \in Q \smallsetminus E_1 \text{ has a neighborhood that meets }  \\
Q^*_\nu \text{ for at most } C \text{ distinct } Q_\nu.\label{22D}
\end{array}\\
\begin{array}{c}\text{Each } \theta_\nu \text{ is a non-negative continuous function on } Q,  \\
\text{vanishing outside } Q \cap Q_\nu^*.\end{array}\label{23D}\\
\sum_\nu\limits \theta_\nu(x) = 1 \quad \text{ if } x \in Q \smallsetminus E_1, \quad 0 \text{ if } x \in E_1.\label{24D}
\end{eqnarray}
Thanks to~(\ref{19}), we can pick points 
$ x_\nu \in E_1$
  such that 
\begin{equation}\label{26D}
\dist (x_\nu, Q^*_\nu) \leq C \delta_\nu.
\end{equation}
We next prove a continuity property of the fibers $v(x) + H^0_x$. 

\begin{lem}\label{p20.lem} Given $x \in E_1$ and $\epsilon > 0$, there exists
 $\delta > 0$ for which the following holds. 
Let $Q_\nu$ be a Whitney cube such that distance $(x, Q^*_\nu) < \delta$.  
 Then 
\begin{enumerate} 
\item \label{27D}
$ \lvert v(x) - v(x_\nu) \rvert < \epsilon $, and 
\item \label{28D}
$\dist (v(x), v(y) + H^0_y) < \epsilon $ for all $ y \in Q^*_\nu \cap Q.$
\end{enumerate}
\end{lem}
 
 {\it Proof}: Fix $x \in E_1$ and $\epsilon > 0$. Let $\delta > 0$ be a small enough number, to be picked later. Let $Q_\nu$ be a Whitney cube such that  
\begin{equation}\label{29D}
\dist (x, Q^*_\nu)< \delta.
\end{equation}
Then, by~(\ref{19}), we have 
\begin{equation}\label{30D}
\delta_\nu \leq \dist (E_1, Q^*_\nu) \leq \dist (x, Q^*_\nu)< \delta,
\end{equation}
hence,~(\ref{26D}) and~(\ref{29D}) yield the estimates  
\begin{equation}\label{31D}
\begin{split}
\lvert x - x_\nu \rvert \ \leq \ \dist (x, Q^*_\nu) + \text{ diameter } (Q^*_\nu) + \dist (Q^*_\nu, x_\nu)  
\leq \delta + C \delta_\nu \leq C^\prime \delta.
\end{split}
\end{equation}
Since $x$ and $x_\nu$ belong to $E_1$,~(\ref{31D}) implies~(\ref{27D}) thanks to~(\ref{18D}), provided we take $\delta$ small enough. Also, for any $y \in Q^*_\nu \cap  Q$, we learn from ~(\ref{29D}),~(\ref{30D}) that  

  $\lvert y - x \rvert \ \leq \text{ diameter } (Q^*_\nu) + \dist(x, Q^*_\nu)< C \delta_\nu + C \delta \leq C^\prime \delta$.  

Since the bundle $(v(z) + H^0_z)_{z \in Q}$ is Glaeser stable, it follows that~(\ref{28}) holds, provided we take $\delta$ small enough. 

We now pick $\delta > 0$ small enough that the above arguments go through. Then~(\ref{27}) and~(\ref{28}) hold. The proof of Lemma \ref{p20.lem}  is complete. \qed

We return to the proof of Theorem  \ref{p18.thm}. For each Whitney cube $Q_\nu$, we prepare to apply our inductive hypothesis~(H1) to the family of affine subspaces  
\begin{equation}\label{32D}
\cH_\nu = (v(y) - v(x_\nu) + H^0_y)_{y \in Q^*_\nu \cap Q}.
\end{equation}
Since $Q^*_\nu \cap Q$ is a closed rectangular box, but not necessarily a cube, it may happen that~(\ref{32D}) fails to be a bundle. The cure is simply to fix an affine map $\rho_\nu: \mathbb{R}^n \to \mathbb{R}^n$, such that $\rho_\nu(Q^o) = Q^*_\nu \cap Q$, where $Q^o$ denotes the unit cube. 

The family of affine spaces 
\begin{equation}\label{33D}
\check{\cH}_\nu = \big( v(\rho_\nu \check{y}) - v(x_\nu) + H^0_{\rho_\nu \check{y}} \big)_{\check{y} \in Q^o}
\qtq{is then a bundle.} 
\end{equation}
We write~(\ref{32D}) in the form
\begin{equation}\label{34D}
\cH_\nu = \big( v_\nu(y) + H^0_y \big)_{y \in Q^*_\nu \cap Q}, \text{ where} 
\end{equation}
\begin{equation}\label{35D}
v_\nu(y) \bot H^0_y \text{ for each } y \in Q^*_\nu \cap Q.
\end{equation}
The vector $v_\nu(y)$ is given by 
\begin{equation}\label{36D}
v_\nu(y) = \Pi_y v(y) - \Pi_y v(x_\nu) \text{ for } y \in Q^*_\nu \cap Q, \text{ where}
\end{equation}
$\Pi_y$ denotes the orthogonal projection from $\mathbb{R}^r$ onto the orthocomplement of $H^0_y$. 

Passing to the bundle $\check{\cH}_\nu$, we find that  
\begin{equation}\label{37D}
\check{\cH}_\nu = ( \check{v}_\nu (\check{y}) + H^0_{\rho_\nu \check{y}})_{\check{y} \in Q^o}, \text{ with}
\end{equation}
\begin{equation}\label{38D}
\check{v}_\nu (\check{y}) \bot H^0_{\rho_\nu \check{y}} \text{ for each } \check{y} \in Q^o.
\end{equation}
Here, $\check{v}_\nu (\check{y})$ is given by 
\begin{equation}\label{39D}
\check{v}_\nu (\check{y}) = v_\nu(\rho_\nu \check{y}).
\end{equation}
It is easy to check that $\check{\cH}_\nu$ is a Glaeser stable bundle with non-empty fibers. Moreover, from~(\ref{12}) and~(\ref{20}), we see that the function $y \mapsto \dim H^0_y$ takes at most $\Lambda - 1$ values as $y$ ranges over $Q^*_\nu \cap Q$. Therefore, the bundle $\check{\cH}_\nu$ has at most 
$\Lambda - 1$ strata.

Thus, our inductive hypothesis~(\ref{11}) applies to the bundle $\check{\cH}_\nu$. Consequently, we obtain the following results for the family of affine spaces $\cH_\nu$. 

We obtain sets 
\begin{equation}\label{40D}
S_\nu (x) \subset Q^*_\nu \cap Q \quad \text{ for each } x \in Q^*_\nu \cap Q,
\end{equation}
   and linear maps 
\begin{equation}\label{41D}
A_\nu (x,y) \ : \ \mathbb{R}^r \to \mathbb{R}^r  \text{ for each } x \in Q^*_\nu \cap Q, y \in S_\nu(x).
\end{equation}
\begin{equation}\label{42D}
\text{The sets } S_\nu(x) \text{ each contain at most } C \text{ points.}
\end{equation}
\begin{equation}\label{43D}
\text{The } S_\nu(x) \text{ and } A_\nu(x,y) \text{ are determined by } (H^0_z)_{z \in Q^*_\nu \cap Q}.
\end{equation}
Moreover, setting  
\begin{equation}\label{44D}
F_\nu(x) = \sum_{y \in S_\nu (x)}\limits \ A_\nu(x,y) v_\nu(y) \quad \text{ for } \quad x \in Q^*_\nu \cap Q ,
\end{equation}
we find that 
\begin{equation}\label{45D}
F_\nu \text{ is continuous on }  Q^*_\nu \cap Q ,
\end{equation}
\begin{equation}\label{46D}
F_\nu(x) \in v_\nu (x) + H^0_x = v(x) - v(x_\nu) + H^0_x \quad \text{ for each } \ \ x \in Q^*_\nu \cap Q ,
\end{equation}
and
\begin{equation}\label{47D}
\max_{\quad x \in Q^*_\nu \cap Q}\limits \quad \bigm| F_\nu(x) \bigm| \ \leq C \sup_{\quad y \in Q^*_\nu \cap Q} \bigm| v_\nu(y) \bigm|.
\end{equation}
Let us estimate the right-hand side of~(\ref{47D}). For any $Q_\nu$, formula~(\ref{36D}) shows that  
\begin{equation}\label{48D}
\sup_{\quad y \in Q^*_\nu \cap Q} \bigm| v_\nu(y) \bigm| \ \leq 2  \sup_{\quad y \in Q}  \bigm| v(y) \bigm|.
\end{equation}
Moreover, let $x \in E_1, \ \epsilon > 0$ be given, and let $\delta$ be as in 
Lemma \ref{p20.lem}. Given any $Q_\nu$ such that distance $(x, Q^*_\nu) < \delta$, and given any $y \in Q^*_\nu \cap Q$, Lemma  \ref{p20.lem} tells us that 

  $\bigm| v(x) - v(x_\nu) \bigm| \ < \ \epsilon$  
  and 
distance $(v(x), v(y) + H^0_y) < \epsilon$. 

  Consequently,  
\begin{equation}\label{49D}
\dist (0, v(y) - v(x_\nu) + H^0_y) < 2 \epsilon \text{ and } \bigm| v(x) - v(x_\nu) \bigm| < \epsilon.
\end{equation}
From~(\ref{32D}),~(\ref{34D}),~(\ref{35D}), we see that $v_\nu(y)$ is the shortest vector in $v(y) - v(x_\nu) + H^0_y$.  Hence,~(\ref{49D}) yields the estimate  
  $\bigm| v_\nu(y) \bigm| \ < 2 \epsilon$. 

Therefore, we obtain the following result. 
 Let $ x \in E_1 $ and $ \epsilon > 0 $ be given. Let $ \delta $ be as in 
Lemma \ref{p20.lem}. Then, for any $ Q_\nu $ such that distance $ (x, Q^*_\nu) < \delta,$
 we have 
\begin{equation}\label{50D}
\sup_{\quad y \in Q^*_{\nu} \cap Q}\limits \quad \bigm| v_\nu(y) \bigm| \ \leq \ 2 \epsilon , \ \text{ and } \ \bigm| v(x) - v(x_\nu) \bigm| < \epsilon.
\end{equation}
From~(\ref{47D}),~(\ref{48D}),~(\ref{50D}), we see that  
\begin{equation}\label{51D}
\max_{\qquad x \in Q^*_\nu \cap \ Q}\limits \ \bigm| F_\nu(x) \bigm| \leq C \sup_{\qquad y \in Q} \bigm| v(y) \bigm|
\end{equation}
 for each $\nu$, and that the following holds. 
 Let $ x \in E_1 $ and $ \epsilon > 0 $ be given.
 Let $ \delta  $ be be as in Lemma \ref{p20.lem}, and let   
$y \in Q^*_\nu \cap Q \cap B(x, \delta)$.  Then 
 \begin{equation}\label{52D}
  \bigm|  F_\nu(y)  \bigm| \ \leq C \epsilon , \text{ and }  \bigm|  v(x) - v(x_\nu) \bigm| < \epsilon.
\end{equation}
We now define a map $F: Q \to \mathbb{R}^r$, by setting  
\begin{equation}\label{53D}
F(x) = v(x) \text{ for } x \in E_1, \text{ and }
\end{equation}
\begin{equation}\label{54D}
F(x) = \sum_\nu\limits \ \theta_\nu(x) \ \cdot \   F_\nu(x) + v(x_\nu)   \  \text{ for } x \in Q \smallsetminus E_1.
\end{equation}
Note that~(\ref{54D}) makes sense, because the sum contains finitely many non-zero terms, and because $\theta_\nu = 0$ outside the set where $F_\nu$ is defined. 

  We will show that $F$ is given in terms of the $(v(y))_{y \in Q}$ by a formula of the form~(\ref{4D}), and that conditions~(\ref{5D}) $\cdots$~(\ref{9D}) are satisfied. As we noted just after~(H2),
 this will complete our induction on $\Lambda$, and establish Theorem  \ref{p18.thm}. 

First, we check that our $F(x)$ is given by~(\ref{4D}), for suitable $S(x), A(x,y)$. We proceed by cases. If $x \in E_1$, then already~(\ref{53D}) has the form~(\ref{4D}), with  
\begin{equation}\label{55D}
S(x) = \{ x \} \text{ and } A(x, y) = \text{ identity.}
\end{equation}
Suppose $x \in Q \smallsetminus E_1$. Then $F(x)$ is defined by~(\ref{54D}). 

  Thanks to~(\ref{23D}), we may restrict the sum in~(\ref{54D}) to those $\nu$ such that $x \in Q^*_\nu$. For each such $\nu$, we substitute~(\ref{36D}) into~(\ref{44D}), and then substitute the resulting formula for $F_\nu(x)$ into~(\ref{54D}). We find that  
\begin{equation}\label{56D}
F(x) = \sum_{Q_\nu^* \ni x}\limits \ \theta_\nu(x) \ \cdot \   v(x_\nu) + \sum_{y \in S_\nu (x)}\limits \ A_\nu(x,y) \ \cdot \ \big( \Pi_y v(y) - \Pi_y v(x_\nu) \big)  
\end{equation}
 
which is a formula of the form~(\ref{4}). 

Thus, in all cases, $F$ is given by a formula~(\ref{4}). Moreover, examining~(\ref{55D}) and~(\ref{56D}) (and recalling~(\ref{40D}) $\cdots$~(\ref{43D})
as well as~(\ref{22})), we see that~(\ref{5}),~(\ref{6}),~(\ref{7}) hold, and that in our formula~(\ref{4}) for $F$, each $S(x)$ contains at most $C$ points. Thus~(\ref{9}) holds, with a suitable $d_\Lambda$ in place of $d$. 

It remains to prove~(\ref{8}), and to show that our $F$ is a section of the bundle $\cH$. Thus, we must establish the following. 
\begin{eqnarray}
F: Q \to \mathbb{R}^r \text{ is continuous.}\label{57D}\\
F(x) \in \ v(x) + H^0_x \text{ for each } x \in Q.\label{58D}\\
\big| F(x) \big| \leq C \sup_{y \in Q} \big| v(y) \big| \text{ for each } x \in Q.\label{59D}
\end{eqnarray}
 
The proof of Theorem  \ref{p18.thm} is reduced to proving~(\ref{57D}),~(\ref{58D}),~(\ref{59D}). 

Let us prove~(\ref{57D}). Fix $x \in Q$; we show that $F$ is continuous at $x$. If $x \notin E_1$, then~(\ref{22D}),~(\ref{23D}),~(\ref{45D}) and~(\ref{54D}) easily imply that $F$ is continuous at $x$. 

On the other hand, suppose $x \in E_1$. To show that $F$ is continuous at $x$, we must prove that 
\begin{equation}\label{60D}
\lim_{y \to x,  y \in E_1}
 v(y) = v(x)\qtq{and that}
\end{equation}
\begin{equation}\label{61D}
\lim_{y \to x,  y \in Q \smallsetminus E_1}
      \ \ \sum_\nu\limits \ \theta_\nu(y)   F_\nu(y) + v (x_\nu)   = v(x).
\end{equation}
We obtain~(\ref{60D}) as an immediate consequence of~(\ref{18D}). To prove~(\ref{61D}), we bring in~(\ref{52D}). Let $\epsilon > 0$, and let $\delta >0$ arise from $ \epsilon, x$ as in~(\ref{52D}). Let $y \in Q \smallsetminus E_1$; and suppose $\big| y - x \big| < \delta$. For each $\nu$ such that $y \in Q^*_\nu$,~(\ref{52D}) gives 
\begin{equation}\label{62D}
\big| \theta_\nu(y) \ \cdot \ [F_\nu(y) + v(x_\nu) - v(x) ] \big| \leq C \epsilon \theta_\nu(y).
\end{equation}
For each $\nu$ such that $y \notin Q_\nu^*$,~(\ref{62D}) holds trivially, since $\theta_\nu(y) = 0$. Thus,~(\ref{62D}) holds for all $\nu$. 
Summing on $\nu$, and recalling~(\ref{24D}), we conclude that  

  $\big| \sum_\nu\limits \ \theta_\nu(y) \ \cdot \   F_\nu(y) + v(x_\nu)   - v(x) \big| \leq C \epsilon$. 

This holds for any $y \in Q \smallsetminus E_1$ such that $\big| y - x \big| < \delta$. The proof of~(\ref{61D}) is complete. Thus,~(\ref{57D}) is now proven. 

To prove~(\ref{58D}), we again proceed by cases. If $x \in E_1$, then~(\ref{58D}) holds trivially, by~(\ref{53D}). On the other hand, suppose $x \in Q \smallsetminus E_1$. Then~(\ref{46D}) gives $[ F_\nu(x) + v(x_\nu) ] \in \ v(x) + H^0_x \quad $ for each $\nu$ such that $Q_\nu^* \ni x$. 

  Since also $\theta_\nu(x) = 0 \quad $ for $x \notin Q^*_\nu$, and since $\sum_\nu\limits \theta_\nu(x) = 1$, 
  it follows that 

  $\sum_\nu\limits \theta_\nu(x) \ \cdot \ [F_\nu(x) + v(x_\nu)] \ \in \ v(x) + H^0_x$, i.e., 

  $F(x) \in \ v(x) + H^0_x$. Thus,~(\ref{58D}) holds in all cases. 

Finally, we check~(\ref{59D}). For $x \in E_1$,~(\ref{59D}) is trivial from the definition~(\ref{53D}). On the other hand, suppose $x \in Q \smallsetminus E_1$. For each $\nu$ such that $Q^*_\nu \ni x$,~(\ref{51D}) gives
\begin{equation}\label{63D}
\big| \theta_\nu(x) \ \cdot \ [F_\nu(x) + v(x_\nu)] \big| \leq C \theta_\nu(x) \cdot \sup_{\quad y \in Q} \big| v(y) \big|.
\end{equation}
Estimate~(\ref{63D}) also holds trivially for $x \notin Q^*_\nu$, since then $\theta_\nu(x) = 0$. 
  Thus,~(\ref{63D}) holds for all $\nu$. Summing on $\nu$, we find that 
$$
\big| F(x) \big| \leq \sum_\nu \big| \theta_\nu(x)  \cdot  [F_\nu(x) + v(x_\nu)] \big| \ \leq \ C \sup_{ y \in Q} \big| v(y) \big|  \cdot \sum_\nu \theta_\nu(x) 
  = C \sup_{y \in Q} \big| v(y) \big|,
$$
 thanks to~(\ref{24D}) and~(\ref{54D}). 

  Thus~(\ref{59D}) holds in all cases. The proof of Theorem  \ref{p18.thm} is complete.\qed
\medskip

  Let $\widetilde{F}$ be any section of the bundle $\cH$ in Theorem  \ref{p18.thm}. For each $x \in Q$, we have $\lvert v (x) \rvert \leq \lvert \widetilde{F} (x) \rvert$, since $\widetilde{F} (x) \in v (x) + H^0_x$ and $v (x) \bot H^0_x$. Therefore, the section $F$ produced by Theorem  \ref{p18.thm} satisfies the estimate  
$\max_{x \in Q} \lvert F (x) \rvert \leq C \cdot \max_{x \in Q} \lvert \widetilde{F} (x) \rvert$, where $C$ depends only on $n, r$.

\subsection{Computing the section of a bundle}{\ }
 \label{subsection.3.5}

Here, we combine our results from the last few sections. 
  Let 
\begin{equation}\label{1E}
 \cH = (v(x) + H^0_x)_{x \in Q} \text{ be a bundle, where}
\end{equation}
\begin{equation}\label{2E}
 \cH^0 = (H^0_x)_{x \in Q} \text{ is a homogeneous bundle, and}
\end{equation}
\begin{equation}\label{3E}
 v(x) \bot H^0_x \quad \text{ for each } x \in Q.
\end{equation}
  Suppose $\cH$ has a section. Then the iterated Glaeser refinements of $\cH$ have non-empty fibers, and may therefore be written as 
\begin{equation}\label{4E}
 \cH^\ell = (v^\ell (x) + H^{0, \ell}_x)_{x \in Q} \text{ where}
\end{equation}
\begin{equation}\label{5E}
 \cH^{0, \ell} = (H^{0, \ell}_x)_{x \in Q} \text{ is a homogeneous bundle, and}
\end{equation}
\begin{equation}\label{6E}
 v^\ell (x) \bot H^{0, \ell}_x \quad \text{ for each } x \in Q.
\end{equation}
  Let $\xi^\ell_i (x) \in \mathbb{R}, \ y^{\ell, \nu}_i (x) \in Q, \ \beta^{\ell , \nu}_{ij} (x) \in \mathbb{R}, \ w_i (x) \in \mathbb{R}^r$ be as in section 
\label{subsection.3.2}. Thus, 
\begin{equation}\label{7E}
 \xi^0_i (x) = i^{th} \text{ component of } v(x), \text{ for } x \in Q;
\end{equation}
\begin{equation}\label{8E}
 \xi^\ell_i (x) = \lim_{\quad \nu \to \infty} \ \sum^r\limits_{j=1}\limits \ \beta^{\ell , \nu}_{ij} (x) \ \xi^{\ell - 1}_j \big( y^{\ell, \nu}_i (x) \big)
\end{equation}
 for $ x \in Q,  
1 \leq \ell \leq 2r + 1, \quad 1 \leq i \leq r$,
  and 
\begin{equation}\label{9E}
 v^{2r+1} (x) = \sum^r\limits_{i=1}\limits \ \xi^{2r + 1}_i (x) w_i (x) \quad \text{ for } x \in Q.
\end{equation}
  Recall that  
$\beta_{ij}^{\ell , \nu} (x), y^{\ell , \nu}_i (x) $    and $ w_i (x) $ are determined by the homogeneous 
bundle $ \cH^0, $ independently of the vectors $ (v(z))_{z \in Q}.$
The bundle $\cH^{2r+1} = (v^{2r+1}(x) + H^{0, 2r+1}_x)_{x \in Q}$ is Glaeser stable, with non-empty fibers. Hence, the results of section  \ref{subsection.3.4} 
apply to $\cH^{2r+1}$. 
Thus, we obtain a section of $\cH^{2r+1}$ of the form 
\begin{equation}\label{11E}
F(x) = \sum_{y \in S (x)}\limits \ A (x,y) v^{2r+1} (y) \qquad (\text{all } x \in Q),
\end{equation}
 where 
$S (x) \subset Q $ and $ \# (S (x) ) \leq d \quad $ for each $ x \in Q $; and 
$ A (x,y): \mathbb{R}^r \to \mathbb{R}^r $ is a linear map,   for each $ x \in Q, y \in S (x).$
  Our section $F$  satisfies the estimate 
\begin{equation}\label{13E}
\max_{\quad x \in Q} \big| F(x) \big| \leq C \max_{\quad x \in Q} \big| \widetilde{F} (x) \big| , \text{ for any section }  \widetilde{F} \text{ of } \cH^{2r + 1}.
\end{equation}
Here, $ d $ and $ C $ depend only on $ n $ and $ r; $ and 
the $ S(x) $  and $ A(x,y) $ are determined by $ \cH^{0, 2r + 1}$, 
independently of the vectors $ v^{2r+1}(z) \ (z \in Q)$.
 
Recall that the bundles $\cH$ and $\cH^{2r+1}$have the same sections. Therefore, substituting~(\ref{9E}) into~(\ref{11E}), and setting 
\begin{equation}\label{15E}
A_i(x,y) = A(x,y) w_i(y) \ \in \mathbb{R}^r \quad \text{ for } x \in Q, \ y \in S(x), \ i = 1, \cdots, r,
\end{equation}
  we find that 
\begin{equation}\label{16E}
F(x) = \sum_{y \in S(x)}\limits \ \sum^r\limits_1\limits \ \ \xi^{2r+1}_i (y) A_i(x,y) \text{ for all } x \in Q.
\end{equation}
  Moreover, 
$F $ is a section of $ \cH$, and
\begin{equation}\label{18E}
\max_{\quad x \in Q} \big| F(x) \big| \leq C \max_{\quad x \in Q} \big| \widetilde{F}(x) \big| \text{ for any section } \widetilde{F} \text{ of } \cH.
\end{equation}
  Furthermore  
 The $ A_i(x,y) $ are determined by $ \cH^0,$ independently   
of the family of vectors $(v(z))_{z \in Q}$.
 
 Thus, we can compute a section of $\cH$ by starting with~(\ref{7E}), then computing the $\xi^\ell_i (x)$ using the recursion~(\ref{8E}), and finally applying~(\ref{16E}) once we know the $\xi^{2r+1}_i (x)$. In particular, we guarantee that the limits in~(\ref{8E}) exist. Here, of course, we make essential use of our assumption that $\cH$ has a section.    

\subsection{Computing a continuous solution of linear equations}{\ }\label{subsection.3.6}

We apply the results of the preceding section, to find continuous solutions of  
\begin{equation}\label{1F}
\phi_1 f_1 + \cdots + \phi_r f_r = \phi \ \text{ on } Q.
\end{equation}
 
  Such a solution $(\phi_1,  \cdots , \phi_r)$ is a section of the bundle 
\begin{equation}\label{2F}
\cH = (H_x)_{x \in Q}, \text{ where}
\end{equation}
\begin{equation}\label{3F}
H_x = \{v = (v_1, \cdots , v_r) \in \mathbb{R}^r: \ v_1 f_1 (x) + \cdots + v_r f_r (x) = \phi (x) \}.
\end{equation}
  We write $\cH$ in the form
\begin{equation}\label{4F}
\cH = (v(x) + H^0_x)_{x \in Q} , \text{ where}
\end{equation}
\begin{equation}\label{5F}
H^0_x = \{v = (v_1, \cdots , v_r) \in \mathbb{R}^r: v_1 f_1 (x) + \cdots + v_r f_r (x) = 0 \},\text{ and}
\end{equation}
\begin{equation}\label{6F}
v(x) = \phi (x) \cdot (\widetilde{\xi}_1(x), \cdots , \widetilde{\xi}_r(x)); \text{ here,}
\end{equation}
\begin{equation}\label{7F}
\widetilde{\xi}_i(x) = 
\left\{
\begin{array}{l} 
\quad 0 \quad  \qtq{if} f_1(x) = f_2(x) = \cdots = f_r(x) = 0\\
{f_i(x)}/\bigl(f_1^2(x) + \cdots + f_r^2(x)\bigr)
\qtq{otherwise.}
\end{array}
\right.
\end{equation}
Note that 
\begin{equation}\label{9F}
v(x) \bot H^0_x \ \ \text{ for each } x \in Q.
\end{equation}
 
Specializing the discussion in the preceding section to the bundle 
  ~(\ref{4E}) $\cdots$~(\ref{8E}), we obtain the following objects:
\begin{itemize}
 \item coefficients $\beta^{\ell , \nu}_{ij} (x) \in \mathbb{R}, \ $ for $x \in Q, \ 1 \leq \ell \leq 2r+1, \ \nu \geq 1, 1 \leq i, j \leq r$;

 \item points $y^{\ell , \nu}_i (x) \in Q, \quad \text{ for } x \in Q, \ 1 \leq \ell \leq 2r+1, \ \nu \geq 1, \ 1 \leq i \leq r$;

\item finite sets $S(x) \subset Q, \quad$ for $x \in Q$; and

\item vectors $A_i (x,y) \in \mathbb{R}^{r}, \quad$ for $x \in Q, y \in S(x), \ 1 \leq i \leq r$.
\end{itemize}

These objects depend only on the functions $f_1, \cdots , f_r$. 

We write $A_{ij}(x,y)$ to denote the $i^{th}$ component of the vector $A_j (x,y)$. 

To attempt to solve equation~(\ref{1F}), we use the following 

\begin{procedure} \label{10F}
 First, compute $ \xi^\ell_i (x) \in \mathbb{R}, $ for all $ x \in Q, \ \ 0 \leq \ell \leq 2r+1, \ 1 \leq i \leq r,$ 
by the recursion:
\begin{eqnarray}
\xi^0_i (x) &=& \widetilde{\xi}_i (x) \cdot \phi (x) \quad \text{ for } \ 1 \leq i \leq r; \ \text{ and}\label{11F}
\\
\xi^\ell_i (x) &=& \lim_{\ \nu \to \infty}  \textstyle{ \sum^r_{j=1}} \ \beta^{\ell , \nu}_{ij} (x) \cdot \xi^{\ell - 1}_j (y^{\ell , \nu}_i (x))
\label{12F}
\end{eqnarray}
 for $ \ 1 \leq i \leq r,  
1 \leq \ell \leq 2r+1.$

Then define functions $\Phi_1, \cdots , \Phi_r: Q \to \mathbb{R} \ $, by setting
\begin{equation}\label{13F}
\Phi_i (x) = \sum_{y \in S (x)}\limits \ \sum^r\limits_{j=1}\limits \ A_{ij} (x, y) \cdot \xi^{2r+1}_j  (y)\qtq{for} \ \ x \in Q, \ \ 1 \leq i \leq r
\end{equation}

If, for some $x \in Q$ and $i = 1, \cdots , r$, the limit in~(\ref{12F}) fails to exist, then our procedure~(\ref{10F}) fails. Otherwise, procedure~(\ref{10F}) produces functions $\Phi_1, \cdots , \Phi_r: Q \ \to \mathbb{R}$. These functions may or may not be continuous. 
\end{procedure}

The next result follows at once from the discussion in the preceding section. It tells us that, if equation~(\ref{1E}) has a continuous solution, then procedure~(\ref{10F}) produces an essentially optimal continuous solution of~(\ref{1E}). 

\begin{thm}\label{p26.thm}
\begin{enumerate}
 \item The objects $\widetilde{\xi}_i (x), \beta^{\ell , \nu}_{ij} (x), y^{\ell , \nu}_i (x), S(x), \text{ and } A_{ij} (x,y)$, used in procedure~(\ref{10F}), depend only on $f_1, \cdots , f_r$, and not on the function $\phi$.

 \item For each $x \in Q$, the set $S (x) \subset Q$ contains at most $d$ points, where $d$ depends only on $n$ and $r$.

\item Let $\phi : Q \to \mathbb{R}$, and let $\phi_1, \cdots , \phi_r : Q \to \mathbb{R}$ be continuous functions 
such that $\quad \phi_1 f_1 + \cdots + \phi_r f_r = \phi$ on $Q. \ \ $Then procedure~(\ref{10F}) succeeds, the resulting functions $\Phi_1, \cdots , \Phi_r: Q \to \mathbb{R}$ are continuous, and  
$\ \Phi_1 f_1 + \cdots + \Phi_r f_r = \phi$ on $Q$. 
Moreover,
$$ 
\max_{\begin{subarray}{1}
        {x \in Q}\\
1 \leq i \leq r
       \end{subarray} }
 \qquad \big| \Phi_i (x) \big| \leq C \cdot  \notag
\max_{\begin{subarray}{1}
\ {x \in Q}\\
1 \leq i \leq r
	\end{subarray} }
 \big| \phi_i (x) \big| 
$$
where $C$ depends only on $n, r$. 
\end{enumerate}
\end{thm}

For particular functions $f_1 , \cdots , f_r$, it is a tedious, routine exercise to go through the arguments in the past several sections, and compute the $\widetilde{\xi}_i (x), \beta^{\ell , \nu}_i (x), y^{\ell , \nu}_i (x), S(x)$ and $A_{ij} (x,y)$ used in our Procedure~(\ref{10F}). We invite the reader to carry this out for the case of Hochster's equation  \ref{rpwt}.4, and to compare the resulting formulas with those given in Section \ref{sec.3}.

\section{Algebraic geometry approach}

The following simple example illustrates this method.

\begin{exmp}\label{simple.exmp}
Which functions $\phi$ on $\r^2_{xy}$ can be written in the form
\begin{equation}\label{simple.exmp.1}
\phi=\phi_1x^2+\phi_2y^2
\end{equation}
where $\phi_1, \phi_2$ are continuous on $\r^2$?
(We know that the Pointwise Tests (\ref{rpwt}) give an answer in this case,
but the following method will generalize better.)

An obvious necessary condition is that $\phi$ should vanish to
order 2 at the origin. This is, however, not sufficient
since $xy$ can not be written in this form.

To see what happens, we blow up the origin.
The resulting real algebraic variety $p:B_0\r^2\to \r^2$ can be covered by
two charts; one given by coordinates
$x_1=x/y, y_1=y$ the other by coordinates
$x_2=x, y_2=y/x$. Working in the first chart,
pulling back (\ref{simple.exmp.1}) we get the equation
\begin{equation}\label{simple.exmp.2}
\phi\circ p=(\phi_1\circ p)\cdot x_1^2y_1^2+(\phi_2\circ p)\cdot y_1^2.
\end{equation}
The right hand side is divisible by $y_1^2$, so we have our first
condition
\medskip

{\it (\ref{simple.exmp}.1) First test.} Is $( \phi\circ p)/y_1^2$ continuous?
\medskip

If the answer is yes, then we divide by $y_1^2$,
set $\psi:=( \phi\circ p)/y_1^2$ 
and try to solve
\begin{equation}\label{simple.exmp.4}
\psi=\psi_1\cdot x_1^2+\psi_2.
\end{equation}
This  always has a  continuous solution, but we need
a solution where $\psi_i= \phi_i\circ p$ for some
$\phi_i$. Clearly, the $\psi_i$ have to be constant
along the line $(y_1=0)$. This is easily seen to be the only
restriction. We thus set $y_1=0$ and try to solve
\begin{equation}\label{simple.exmp.5}
\psi(x_1, 0)=r_1 x_1^2+r_2\qtq{where $r_i\in \r$.}
\end{equation}
The original 2 variable problem has been reduced to a
1 variable question.
Solvability is easy to decide using either of the following.
\medskip

{\it (\ref{simple.exmp}.2.i) Second test, Wronskian form.}
The following determinant is identically zero
$$
\left|
\begin{array}{ccc}
1 & 1 &  1\\
a^2 & b^2 & c^2 \\
\psi(a, 0) &\psi(b, 0) &\psi(c, 0) 
\end{array}
\right|
$$
\medskip

{\it (\ref{simple.exmp}.2.ii) Second test, finite set form.}
For every $a,b,c\in \r$ there are  $r_i:=r_i(a,b,c)\in \r$
(possibly depending on $a,b,c$)
such that
$$
\psi(a, 0)=r_1a^2+r_2,\quad 
\psi(b, 0)=r_1b^2+r_2\qtq{and}
\psi(c, 0)=r_1c^2+r_2.
$$
(In principle we should check what happens on the second chart,
but in this case it gives nothing new.)
\end{exmp}

Working on $\r^n$, let us now consider the general case
$$
\phi=\tsum_i \phi_i f_i.
$$
As in (\ref{simple.exmp}), we start by blowing up
either the common zero set $Z=(f_1=\cdots=f_r=0)$, or,
what is computationally easier, the ideal
$(f_1,\dots,f_r)$. We get a real algebraic variety
$p:Y\to \r^n$.

Working in various coordinate charts on $Y$,
we get  analogs of the First test (\ref{simple.exmp}.1)
and new equations
$$
\psi=\tsum_i \psi_i g_i.
$$
The solvability again needs to be checked only on an
$(n-1)$-dimensional real algebraic subvariety $Y_E\subset Y$.
One sees, however, that the second tests (\ref{simple.exmp}.2.i--ii)
are both equivalent to the Pointwise tests (\ref{rpwt}),
thus not sufficient in general.

Instead, we focus on what kind of question we
need to solve on $Y_E$. This leads to the 
following  concept.

\begin{defn}\label{desc.prob.efn} A {\it descent problem}  is a compound
object 
$$
{\mathbf D}=\bigl(p:Y\to X, f:p^*E\to F\bigr)
$$ 
 consisting of a 
 proper morphism  of real algebraic varieties $p:Y\to X$,
 an algebraic   vector bundle $E$ on $X$,  an algebraic 
   vector bundle $F$ on $Y$
and an algebraic  vector bundle map
$f:p^*E\to F$.
(See (\ref{r.basic}) for  the basic notions related to
real algebraic varieties.)

Our aim is to understand the image of
$f\circ p^*:C^0(X,E)\to C^0(Y,F)$.
\end{defn}

We have the following analog of (\ref{simple.exmp}.2.ii).

\begin{defn}\label{fin.set.test.defn} 
Let ${\mathbf D}=\bigl(p:Y\to X, f:p^*E\to F\bigr)$
be a descent problem and $\phi_Y\in C^0(Y,F)$. We say that
$\phi_Y$ satisfies the {\it finite set test} if for 
every $y_1,\dots, y_m\in Y$ there is a
$\phi_X=\phi_{X,y_1,\dots, y_m}\in C^0(X,E)$
(possibly depending on $y_1,\dots, y_m$)
such that
$$
\phi_Y(y_i)=f\circ p^*(\phi_X)(y_i)\qtq{for  $i=1,\dots,m$.}
$$
\end{defn}

\begin{defn}\label{fin.det.defn} 
 A  descent problem  ${\mathbf D}=\bigl(p:Y\to X, f:p^*E\to F\bigr)$
is called {\it finitely determined} if
 for every  $\phi_Y\in  C^0(Y, F)$ the following 
are equivalent.
\begin{enumerate} 
\item  $\phi_Y\in \im \bigl[f\circ p^*:C^0(X,E)\to C^0(Y, F)\bigr]$.
\item  $\phi_Y$ satisfies the  finite set test.
\end{enumerate}
\end{defn}

\begin{say}[Outline of the main result]\label{outline}
Our theorem (\ref{main.thm.v2}) 
gives an algorithm to decide the answer to  Question \ref{rq1}.
The precise formulation 
is somewhat technical to state,
so here is a rough explanation of what kind of answer it gives
and what we mean by an ``algorithm.''
There are three main parts.

{\it Part 1.} First, starting with $\r^n$ and $f_1,\dots, f_r$
we construct a finitely determined
descent problem  ${\mathbf D}=\bigl(p:Y\to \r^n, f:p^*E\to F\bigr)$.
This is purely
algebraic, can be effectively carried out and
independent of $\phi$.

{\it Part 2.} There is a partially defined
``twisted pull-back'' map $p^{(*)}: C^0(\r^n)\map C^0(Y,F)$ 
(\ref{rel.desc.prob.sefn})
which  is obtained as an iteration of three kinds of steps
\begin{enumerate}
\item We compose a function  by a
real algebraic map.
\item We create a vector function out of several functions or
decompose a vector function into its coordinate functions.
\item We  choose  local (real analytic) coordinates $\{y_i\}$
and ask if a certain  function of the form
$\psi_{j+1}:=\psi_j\cdot\textstyle{\prod_i} y_i^{-m_{i}}$  
is continuous or not where $m_{i}\in \z$.
 \end{enumerate}
If any of the answers is no,
then the original $\phi$ can not be written as $\sum_i \phi_if_i$
and we are done.
If all the answers are  yes, then
we end up with $p^{(*)}\phi\in C^0(Y,F)$.

{\it Part 3.} We show that $\phi=\sum_i \phi_if_i$
is solvable iff $p^{(*)}\phi\in C^0(Y,F)$
satisfies the finite set test (\ref{fin.set.test.defn}).

By following the proof, one can actually write down solutions
$\phi_i$, but this relies on some artificial choices.
The main ingredient that we need is to choose extensions of
certain  functions defined on closed
semialgebraic subsets to the whole $\r^n$.
In general, there does not seem to be any natural extension,
and we do not know if it makes sense to ask for the ``best possible''
solution or not.
\end{say}

{\it Negative aspects.} There are two difficulties in carrying out
this procedure in any given case.
First, in practice, (3) of Part 2 may not be
effectively doable.
Second,  we may need to compose $\psi_{j+1}$ with a 
real algebraic map $r_{j+1}$
such that $\psi_{j}$ vanishes on the image of $r_{j+1}$.
 Thus we really need to compute limits
and work with the resulting functions.
This also makes it  difficult to interpret our answer
on $\r^n$ directly.

{\it Positive aspects.} 
On the other hand, just knowing that the answer has the above 
general structure already has some useful consequences.


First, the general framework works for other classes of functions;
for instance the same algebraic set-up also applies
in case $\phi$ and the $\phi_i$ are H\"older continuous.

Another
consequence we obtain is that if $\phi=\sum_i \phi_if_i$ is solvable and
$\phi$ has certain additional properties,
then one can also find a solution $\phi=\sum_i \psi_if_i$
where the $\psi_i$ also have these   additional  properties.
We list two  such examples below; see also (\ref{foon.on}).
For the proof, see  (\ref{pf.of.main.thm.induct}) and (\ref{pf.of.r.main.cors}).


\begin{cor} \label{r.main.cors}
Fix $f_1,\dots, f_r$ and assume that
$\phi=\sum_i \phi_if_i$ is solvable. Then:
\begin{enumerate}
\item If $\phi$ is semialgebraic
(\ref{r.basic}) then  there is a solution 
$\phi=\sum_i \psi_if_i$ such that the $\psi_i$ are also  semialgebraic.
\item Let $U\subset \r^n\setminus Z$ be an open set such that
  $\phi$ is  $C^m$ on $U$ for some $m\in \{1,2, \dots, \infty, \omega\}$.
Then there is a solution $\phi=\sum_i \psi_if_i$ such that
the $\psi_i$ are  also $C^m$ on $U$.
\end{enumerate}
\end{cor}

\begin{exmps}\label{counter.exmps} The next series of examples shows 
several possible variants of
 (\ref{r.main.cors}) that fail.



(1) Here $\phi$ is a polynomial, but the $\phi_i$ must have
very small H\"older exponents.

For $m\geq 1$, take $\phi:=x^{2m}+(x^{2m-1}-y^{2m+1})^2$
and $f_1=x^{2m+2}+y^{2m+2}$. There is only one solution,
$$
\phi_1=\frac{x^{2m}+(x^{2m-1}-y^{2m+1})^2}{x^{2m+2}+y^{2m+2}}.
$$
We claim that it is H\"older with exponent $\frac2{2m-1}$.
The exponent is achieved along the curve
$x^{2m-1}-y^{2m+1}=0$, parametrized as
$\bigl(t^{(2m+1)/(2m-1)}, t\bigr)$.

 
(2) Here $\phi$ is $C^n$, there is a $C^0$ solution but no 
H\"older solution.

On 
$[-\frac12, \frac12]\subset \r^1$ set $f=x^n$ and $\phi=x^n/\log|x|$.
Then $\phi$ is $C^n$  and $\phi=\frac1{\log|x|}\cdot f$.
Note that $\frac1{\log|x|}$ is continuous but not H\"older.
(These can be extended to $\r^1$ in many ways.)

(3) Question: If  $\phi$ is $C^{\infty}$ and there is a $C^0$ solution,
 is there always a  H\"older solution?

(4) Let $g(x)$ be a real analytic function. Set $f_1:=y$ and
$\phi:=\sin\bigl(g(x)y\bigr)$. Then 
$\phi_1:=\phi/y$ is also real analytic and
$\phi=\phi_1\cdot f_1$ is the only solution.
Note that $|\phi(x,y)|\leq 1$ everywhere yet
$\phi_1(x,0)=g(x)$ can grow arbitrary fast.

(5) In general
there is no solution $\phi=\sum_i \psi_if_i$ such that
$\supp\psi_i\subset \supp \phi$ for every $i$.
As an example, take $f_1=x^2+x^4, f_2=x^2+y^2$ and 
$$
\phi(x,y)=
\left\{
\begin{array}{cl}
x^4-y^2&\qtq{if $y^2\geq x^4$ and}\\
0&\qtq{if $y^2\leq x^4$.}
\end{array}
\right.
$$
Note that $\phi=f_1-\phi_2f_2$ where
$$
\phi_2(x,y)=
\left\{
\begin{array}{cl}
1&\qtq{if $y^2\geq x^4$ and}\\
\frac{x^2+x^4}{x^2+y^2}&\qtq{if $y^2\leq x^4$.}
\end{array}
\right.
$$
Let $\phi=\phi_1\cdot(x^2+x^4)+\psi_2\cdot(x^2+y^2)$ be any
continuous solution. Setting $x=0$ we get that
$-y^2=\psi_2(0,y)\cdot y^2$, hence $\psi_2(0,0)=-1$.
Thus $\supp \psi_2$ can not be contained in 
$\supp\phi$.

On the other hand, given any solution
$\phi=\sum_i \phi_if_i$, let $\chi$ be a function that is 1 on $\supp\phi$ and
0 outside a small neighborhood of it. Then
 $\phi=\chi\phi=\sum (\chi\phi_i)f_i$. 
Thus we do have solutions whose support 
is close to $\supp \phi$.

\end{exmps}

\subsection{Descent problems and their scions}{\ }

\begin{say}[Basic set-up]\label{r.basic}
From now on, $X$ denotes  a fixed
 real algebraic variety.
We always think of $X$ as the real points of  a complex affine algebraic
 variety  $X_{\c}$ that is  defined 
by real equations. 
(All our algebraic varieties are assumed reduced,
that is, a function is zero iff it is zero at every point).

By a {\it projective variety over $X$}
we mean the real points of a closed 
subvariety  
$Y\subset X\times \c\p^N$.
Every such $Y$ is again
 the set of real points of  a complex affine algebraic
 variety   $Y_{\c}\subset X_{\c}\times \c\p^N$
 that is  defined 
by real equations.
For instance,  $X\times \r\p^N$ is contained 
in the affine variety which is the complement
of the hypersurface $(\sum y_i^2=0)$ where $y_i$ are the coordinates on $\p^N$.
 
A  variety $Y$ over $X$
comes equipped with a  morphism $p:Y\to X$ to $X$,
given by the first projection of $X\times \c\p^N$.
Given such $p_i:Y_i\to X$, a morphism between them is a
morphism of real 
algebraic varieties $\phi:Y_1\to Y_2$ such that
$p_1=p_2\circ \phi$.

Given $p_i:Y_i\to X$, their {\it fiber product} is
$$
Y_1\times_XY_2:=\bigl\{(y_1, y_2): p_1(y_1)=p_2(y_2)\bigr\}
\subset Y_1\times Y_2.
$$
This comes with a natural projection
$p:Y_1\times_XY_2\to X$
and $p^{-1}(x)=p_1^{-1}(x)\times p_2^{-1}(x)$ for every $x\in X$.
(Note, however, that even if the $Y_i$ are smooth,
their fiber product can be very singular.)
If $X$ is irreducible, we are frequently
interested only in those irreducible components that dominate $X$;
called the {\it dominant components}.

$\reg(Y)$ denotes the ring of all regular functions on  $Y$.
These are locally quotients of polynomials
$p(x)/q(x)$ where $q(x)$ is nowhere zero.

By an  algebraic {\it vector bundle} on $Y$
we mean the restriction of a complex algebraic
 vector bundle from $Y_{\c}$ to $Y$.
All such  vector bundles can be given by patching
trivial bundles on a 
Zariski open cover $X=\cup_i U_i$  using  transition functions in
$\reg(U_i\cap U_j)$. 
(Note that the latter condition is not quite equivalent to our definition, 
but this is not important for us, cf.\ \cite[Chap.12]{bcr}.)

 Note that there are two natural topologies on  a real algebraic variety $Y$,
the Euclidean topology and the Zariski topology.
The closed sets of the latter are exactly the
closed subvarieties of  $Y$.
A Zariski closed (resp.\ open) subset of $Y$ is also
 Euclidean  closed (resp.\ open).

A closed {\it basic semialgebraic subset} of $Y$ is defined by
finitely many inequalities $g_i\geq 0$. Using finite intersections and
complements we get all semialgebraic subsets.
A function is semialgebraic iff its graph is semialgebraic.
See  \cite[Chap.2]{bcr} for  a detailed treatment.
\end{say}

We need various ways of modifying descent problems.
The following  definition is chosen to 
consist of simple and computable steps yet
be broad enough
for the proofs to work.
(It should become clear that several variants of the definition would 
also work. We found the present one convenient to use.)

\begin{defn}[Scions of  descent problems]\label{rel.desc.prob.sefn}
Let ${\mathbf D}=\bigl(p:Y\to X, f:p^*E\to F\bigr)$  be a descent problem.
A {\it scion} of  ${\mathbf D}$
is any  descent problem 
${\mathbf D}_s=\bigl(p_s:Y_s\to X, f_s:p_s^*E\to F_s\bigr)$
that can be obtained by  repeated application of
the following procedures.
\begin{enumerate}
\item   For a proper morphism $r:Y_1\to Y$ set
$$
r^*{\mathbf D}:=
\bigl(p\circ r:Y_1\to X, r^*f: (p\circ r)^*E\to r^*F\bigr).
$$
As a special case, if $Z\subset X$ is a closed subvariety then the scion 
${\mathbf D}_Z=\bigl(p_Z:Y_Z\to Z, f_Z:p_Z^*(E|_Z)\to F|_{Y_Z}\bigr)$ 
(where $Y_Z:=p^{-1}(Z)$)
is called  the {\it restriction} of  ${\mathbf D}$ to   $Z$.
\item  Given $Y_w$, assume that there are several  proper morphisms 
$r_i:Y_w\to Y$  such that the composites
$p_w:=p\circ r_i$ are all the same. Set
$$
(r_1,\dots, r_m)^*{\mathbf D}:=
\bigl(p_w:Y_w\to X, \tsum_{i=1}^mr_i^*f: p_w^*E\to \tsum_{i=1}^mr_i^*F\bigr)
$$
where $\tsum_{i=1}^mr_i^*f$ is the natural diagonal map.
\item Assume that $f$ factors as 
$p^*E\stackrel{q}{\to} F'\stackrel{j}{\into} F$ 
where 
$F'$ is a vector bundle and
$\rank_y j=\rank_y F'$ for all $y$ in  a 
Euclidean dense Zariski open subset  $Y^0\subset Y$.
Then  set
$$
{\mathbf D}':=\bigl(p:Y\to X, f':=q:p^*E\to F'\bigr).
$$
\end{enumerate}
(The choice of $Y^0$  is actually a quite subtle point. 
Algebraic maps have constant rank over a suitable 
Zariski open subset and we want this open set to
determine what happens with an arbitrary continuous function.
This is why $Y^0$ is assumed Euclidean dense, not just
Zariski dense. If $Y$ is smooth, these are equivalent properties,
but not if $Y$ is singular.
As an example, consider
the Whitney umbrella  $Y:=(x^2=y^2z)\subset \r^3$.
Here $Y\setminus (x=y=0)$ is Zariski open  and Zariski dense.
Its Euclidean  closure does not contain the ``handle''
$(x=y=0,\ z<0)$, so it is
 not Euclidean dense.)

Each scion remembers all of its forebears.
That is, two scions are considered the ``same'' only if they
have been constructed by an identical sequence of procedures.
This is  quite important since the vector bundle  $F_s$
 on a scion ${\mathbf D}_s$ does depend on the
whole sequence.

Every scion comes with a {\it structure map}
$r_s:Y_s\to Y$.

If $\phi\in C^0(Y,F)$ then $r^*\phi\in C^0(Y_1,r^*F)$
and $\tsum_{i=1}^mr_i^*\phi\in C^0(Y_w,\tsum_{i=1}^mr_i^*F)$
are well defined. 
In (3) above,  $j:C^0(Y,F')\to C^0(Y,F)$ is an injection,
hence there is at most one $\phi'\in C^0(Y,F')$
such that $j(\phi')=\phi$.
Iterating these, for any scion ${\mathbf D}_s$ of ${\mathbf D}$
with  structure map
$r_s:Y_s\to Y$
we get a partially defined map, called the {\it twisted pull-back},
$$
r_s^{(*)}: C^0(Y,F)\map C^0(Y_s, F_s).
$$
We will need to know which functions
$\phi$ are in the domain of a twisted pull-back map.
A complete answer is given in (\ref{loc.lift.test}).

The twisted pull-back map sits in a commutative square
$$
\begin{array}{ccc}
C^0(Y,F) & \stackrel{r_s^{(*)}}{\map} & C^0(Y_s, F_s)\\
\uparrow && \uparrow \\
C^0(X,E) & = & C^0(X,E).
\end{array}
$$
If the structure map $r_s:Y_s\to Y$ is surjective, then 
$r^{(*)}: C^0(Y,F)\map C^0(Y_s, F_s)$ is injective (on its domain).
In this case,  understanding the image of
$f\circ p^*:C^0(X,E)\to C^0(Y,F)$ is pretty much
equivalent to understanding the image of
$f_s\circ p_s^*:C^0(X,E)\to C^0(Y_s,F_s)$.


\end{defn}

We are now ready to state our main result, first in the inductive form.

\begin{prop}\label{main.thm.induct} Let
${\mathbf D}=\bigl(p:Y\to X, f:p^*E\to F\bigr)$
be a descent problem. Then
 there is a   scion
 ${\mathbf D}_s=\bigl(p_s:Y_s\to X, f_s:p_s^*E\to F_s\bigr)$
with  surjective  structure map $r_s:Y_s\to Y$
and a closed subvariety $Z\subset X$ 
such that $\dim Z< \dim X$
and for every  $\phi\in  C^0(Y, F)$ following 
are equivalent.
\begin{enumerate} 
\item  $\phi\in \im \bigl[f\circ p^*:C^0(X,E)\to C^0(Y, F)\bigr]$.
\item  $r_s^{(*)}\phi$ is defined and
 $r_s^{(*)}\phi\in \im \bigl[f_s\circ p_s^*:C^0(X,E)\to C^0(Y_s, F_s)\bigr]$,
\item 
\begin{enumerate} 
\item $r_s^{(*)}\phi$  satisfies the finite set test
 (\ref{fin.set.test.defn}) and
\item  $\phi|_{Y_Z}\in
 \im \bigl[f_Z\circ p_Z^*:C^0(Z,E|_Z)\to C^0(Y_Z, F_Z)\bigr]$,
where  the scion \linebreak
${\mathbf D}_Z=\bigl(p_Z:Y_Z\to Z, f_Z:p_Z^*(E|_Z)\to F_Z\bigr)$
is  the restriction of  ${\mathbf D}_s$ to   $Z$
(\ref{rel.desc.prob.sefn}.1).
\end{enumerate} 
\end{enumerate} 
\end{prop}

We can now set $X_1:=Z$, ${\mathbf D}_1:={\mathbf D}_Z$
apply (\ref{main.thm.induct})
to ${\mathbf D}_1$ and get a  descent problem 
 ${\mathbf D}_2:=\bigl({\mathbf D}_1\bigr)_Z$.
Repeating this, 
we obtain  descent problems
 ${\mathbf D}_i=\bigl(p_i:Y_i\to X, f_i:p_i^*E\to F_i\bigr)$
such that 
 the dimension of $p_i(Y_i)$ drops at every step.
Eventually we reach the case where $p_i(Y_i)$ consists of points.
Then the  finite set test (\ref{fin.set.test.defn})
gives the complete answer.
The disjoint union of all the $Y_i$ can be viewed as 
a single scion, hence we get the following algebraic answer to
Question \ref{rq1}.

\begin{thm}\label{main.thm.v2} Let
${\mathbf D}=\bigl(p:Y\to X, f:p^*E\to F\bigr)$
be a descent problem. Then it has a finitely determined scion
 ${\mathbf D}_w=\bigl(p_w:Y_w\to X, f_w:p_w^*E\to F_w\bigr)$
with surjective structure map $r_w:Y_w\to Y$.

That is,   for every  $\phi\in  C^0(Y, F)$ following 
are equivalent.
\begin{enumerate} 
\item  $\phi\in \im \bigl[f\circ p^*:C^0(X,E)\to C^0(Y, F)\bigr]$.
\item  The  twisted pull-back $r_w^{(*)}\phi$ is defined and
 it is contained in the image of 
$f_w\circ p_w^*:C^0(X,E)\to C^0(Y_w, F_w)$,
\item The  twisted pull-back $r_w^{(*)}\phi$ is defined and
 satisfies the  finite set test (\ref{fin.set.test.defn}).\qed 
\end{enumerate} 
\end{thm}

The proof of (\ref{main.thm.v2}) works  for  many subclasses of continuous
functions as well. Next we axiomatize the necessary properties and
describe the main examples.

\subsection{Subclasses of continuous functions}{\ }

\begin{ass}\label{function.classes.ass}
For  real algebraic varieties  $Z$  we consider vector subspaces
$C^*\bigl(Z\bigr)\subset C^0\bigl(Z\bigr)$
 that satisfy
the following properties.
\begin{enumerate}
\item (Local property) If $Z=\cup_i U_i$ is an open cover of $Z$
then $\phi\in C^*\bigl(Z\bigr)$ iff
$\phi|_{U_i}\in C^*\bigl(U_i\bigr)$ for every $i$.
\item ($\reg(Z)$-module) If $\phi\in C^*\bigl(Z\bigr)$
and $h\in \reg\bigl(Z\bigr)$ is a regular function
(\ref{r.basic}) then
 $h\cdot \phi\in C^*\bigl(Z\bigr)$.
\item (Pull-back) For every morphism 
$g:Z_1\to Z_2$, composing with 
$g$ maps
$C^*\bigl(Z_2\bigr)$ to $C^*\bigl(Z_1\bigr)$.
\item (Descent property)
Let $g:Z_1\to Z_2$ be a proper,  surjective   morphism, 
 $\phi\in C^0\bigl(Z_2\bigr)$ and assume that
$\phi\circ g\in C^*\bigl(Z_1\bigr)$.
Then $\phi\in C^*\bigl(Z_2\bigr)$.
\item (Extension property) Let $Z_1\subset Z_2$ be a closed 
semialgebraic subset (\ref{sections.semialg.sets}).
Then the twisted pull-back map
$C^*\bigl(Z_2\bigr)\to C^*\bigl(Z_1\bigr) $ is surjective.
\end{enumerate}
Since every  closed 
semialgebraic subset is the image of a proper morphism
(\ref{sections.semialg.sets}),
we can unite (4) and (5) and avoid using
semialgebraic subsets as follows.
\begin{enumerate}
\item[(4+5)] (Strong descent property) 
Let $g:Z_1\to Z_2$ be a proper   morphism and 
 $\psi\in C^*\bigl(Z_2\bigr)$. Then
$\psi=\phi\circ g$ for some  $\phi\in C^*\bigl(Z_2\bigr)$
iff $\psi$ is constant on every fiber of $g$.
\end{enumerate}
The following additional condition 
comparing 2 classes $C_1^*\subset C_2^*$ is also of interest.
\begin{enumerate}\setcounter{enumi}{5}
\item (Division property) 
Let  $h\in \reg\bigl(Z\bigr)$ be any function whose zero set is nowhere 
Euclidean dense.
If $\phi\in C_1^*\bigl(Z\bigr)$
 and 
 $\phi/h\in C_2^*\bigl(Z\bigr)$ then
$\phi/h\in C_1^*\bigl(Z\bigr)$.
\end{enumerate}
\end{ass}

\begin{exmp}\label{function.classes.exmp} 
Here are some natural examples satisfying the
assumptions (\ref{function.classes.ass}.1--5).
\begin{enumerate}
\item $C^0\bigl(Z\bigr)$, the set of all
continuous functions on  $Z$.
\item $C^h\bigl(Z\bigr)$, the set of all locally 
H\"older continuous functions on $Z$.
\item $S^0\bigl(Z\bigr)$, the set of
continuous 
semialgebraic  functions on $Z$.
\end{enumerate}
Moreover, the pairs $S^0\subset C^0$ and $S^0\subset C^h$
both satisfy (\ref{function.classes.ass}.6).
(By contrast, by (\ref{counter.exmps}.2), the pair $C^h\subset C^0$ does not
 satisfy (\ref{function.classes.ass}.6).)
\end{exmp}

\begin{say}[Proof of (\ref{r.main.cors}.1)]\label{pf.of.r.main.cors}
More generally, consider two
 classes $C_1^*\subset C_2^*$ that satisfy
(\ref{function.classes.ass}.1--5) and also
(\ref{function.classes.ass}.6). 
Let ${\mathbf D}$ be a descent problem and $\phi\in C_1^*(Y,F)$.
We claim that if
$\phi=f\circ p^*(\phi_X)$ is solvable with
$ \phi_X\in C_2^*(X,E)$ then it also has a solution
$\phi=f\circ p^*(\psi_X)$ where
$ \psi_X\in C_1^*(X,E)$.

To see this, let ${\mathbf D}_w$ be a scion as in  (\ref{main.thm.v2}).
By our assumption, the twisted pull-back $r_w^{(*)}\phi$  is  in
$C_2^*\bigl(Y_w, F_w)$ and it satisfies the  finite set test.
For the  finite set test it does not matter what type of functions we
work with. Thus we need to show that $r_w^{(*)}\phi$  is  in
$C_1^*\bigl(Y_w, F_w)$.

In a scion construction, this holds for steps as in
(\ref{rel.desc.prob.sefn}.1--2) by (\ref{function.classes.ass}.3).
 The key question is (\ref{rel.desc.prob.sefn}.3). 
The solution given in (\ref{loc.lift.test}) shows that
it is equivalent to (\ref{function.classes.ass}.6).\qed
\end{say}

\begin{say}[$C^*$-valued functions over semialgebraic sets]
\label{sections.semialg.sets}
 Let $S\subset Z$ be a closed semialgebraic subset.
We can think of  $S$ as the  image of a proper  morphism
$g:W\to Z$ (cf.\ \cite[Sec.2.7]{bcr}).
One can define $C^*(S)$ either as the image of
 $C^*(Z)$ in $C^0(S)$ or as the
preimage of $C^*(W)$ under the pull-back by $g$.
By (\ref{function.classes.ass}.4+5), these two are equivalent.

We also have the following 
\begin{enumerate}
\item (Closed patching condition) Let 
$S_i\subset Z$ be  closed semialgebraic subsets.
Let $\phi_i\in C^*(S_i)$ and assume that
$\phi_i|_{S_i\cap S_j}=\phi_j|_{S_i\cap S_j}$ for every $i,j$.

Then there is a  unique $\phi\in C^*\bigl(\cup_iS_i\bigr)$
such that $\phi|_{S_i}=\phi_i$ for every $i$.
\end{enumerate}
To see this, realize each $S_i $ as the  image of some proper  morphism
$g_i:W_i\to Z$. Let $W:=\amalg_i W_i$ be their disjoint union
and $g:W\to Z$ the corresponding morphism.
Define $\psi\in C^*(W)$ by the conditions
$\psi|_{W_i}=\phi_i\circ g_i$.

The patching condition guarantees that
$\psi$ is constant on the fibers of $g$.
Thus, by (\ref{function.classes.ass}.4+5), $\psi=\phi\circ g$
for some $\phi\in C^*\bigl(\cup_iS_i\bigr)$.

These arguments also show that each $C^*(Z)$ is in fact a module over
$S^0(Z)$, the ring of continuous semialgebraic functions.

\end{say}

\begin{defn}[$C^*$-valued sections]\label{basic.props}

By Serre's theorems, every vector bundle on a complex affine  variety 
  can be written as a
quotient bundle of a trivial bundle and also as a
subbundle of a trivial bundle. Furthermore,
every extension of vector bundles splits.
 
Thus,  on a real algebraic variety, 
every algebraic vector bundle  can be written as a
quotient bundle (and a subbundle) of a trivial bundle
and every constant rank map of  vector bundles splits.

Let $F$ be an algebraic  vector bundle on $Z$
and $Z=\cup_i U_i$ an open cover such that
$F|_{U_i}$ is trivial of rank $r$ for every $i$. Let
$$
C^*\bigl(Z, F\bigr)\subset C^0\bigl(Z, F\bigr)
$$
denote the set of those sections $\phi\in C^0\bigl(Z, F\bigr)$
such that
$\phi|_{U_i}\in C^*\bigl(U_i\bigr)^r$ for every $i$.
If $C^*$ satisfies the properties (\ref{function.classes.ass}.1--2),
 this is independent of the
trivializations and the choice of the covering.

If  $C^*$ satisfies the properties (\ref{function.classes.ass}.1--6)
then their natural analogs also hold for $C^*\bigl(Z, F\bigr)$.
This is clear for the properties (\ref{function.classes.ass}.2--4)
and (\ref{function.classes.ass}.6).

In order to check the extension property (\ref{function.classes.ass}.5)
first note that  we have the following.
\begin{enumerate}
\item
 Let  $f:F_1\to F_2$ be a surjection of 
vector bundles. Then 
$f:C^*\bigl(Z, F_1\bigr)\to C^*\bigl(Z, F_2\bigr)$ is surjective. 
\end{enumerate}
Now  let  $Z_1\subset Z_2$ be an closed subvariety
and $F$ a vector bundle on $Z_2$. Write it as  a quotient of a
trivial bundle $\c_{Z_2}^N$. Every section 
$\phi_1\in C^*\bigl(Z_1, F|_{Z_1}\bigr)$ lifts to a
section in $C^*\bigl(Z_1, \c_{Z_1}^N\bigr)$
which in turn extends to  a
section in $C^*\bigl(Z_2, \c_{Z_2}^N\bigr)$ by (\ref{function.classes.ass}.6).
The image of this lift in $C^*\bigl(Z_2, F|_{Z_2}\bigr)$
gives the required lifting of  $\phi_1$.
\end{defn}

\subsection{Local tests and reduction steps}{\ }

Next we consider various  descent problems
whose solution is unique, if it exists.

\begin{say}[Pull-back test] \label{pull-back.test}
Let $g:Z_1\to Z_2$ be a proper surjection of real algebraic varieties.
Let $F$ be a vector bundle on $Z_2$ and
$\phi_1\in C^*\bigl(Z_1, g^*F\bigr)$. When can we write
$\phi_1=g^*\phi_2$ for some  $\phi_2\in C^*\bigl(Z_2, F\bigr)$?

{\it Answer:} By  (\ref{function.classes.ass}.4), such a $\phi_2$ exists
iff $\phi_1$ is constant on every fiber of $g$.  This can be checked
as follows.

Take the fiber product $Z_3:=Z_1\times_{Z_2}Z_1$ with projections
$\pi_i:Z_3\to Z_1$ for $i=1,2$. Note that $F_3:=\pi_1^*g^*F$ is naturally
isomorphic to $\pi_2^*g^*F$. We see that 
 $\phi_1$ is constant on every fiber of $g$ iff
$$
 \pi_1^*\phi_1- \pi_2^*\phi_1\in C^*\bigl(Z_3, F_3\bigr)
\qtq{is identically 0.} 
$$
Note that this solves descent problems 
${\mathbf D}=\bigl(p:Y\to X, f:p^*E\cong F\bigr)$
where $f$ is an isomorphism.
We use two simple cases.
\begin{enumerate}
\item Assume that there  is a closed subset  $Z\subset X$ such that
$p$ induces  an isomorphism $Y\setminus p^{-1}(Z)\to X\setminus Z$
and $\phi_Y\in C^0\bigl(Y, p^*E\bigr)$  vanishes along
$p^{-1}(Z)$. Then there is a $\phi_X\in C^0\bigl(X, E\bigr)$
such that $\phi_Y=p^*\phi_X$ 
(and $\phi_X$  vanishes along $Z$.)
\item Assume that there is a finite group $G$ acting on $Y$
such that $G$ acts transitively on every fiber of
$\bigl(Y\setminus(\phi_Y=0)\bigr)\to X$. 
Then  there is a $\phi_X\in C^0\bigl(X, E\bigr)$
such that $\phi_Y=p^*\phi_X$. 
\end{enumerate}
\end{say}

\begin{say}[Wronskian test]
\label{2.good.cases}
 Let $\phi,f_1,\dots, f_r$ be  functions on a set $Z$.
Assume that the $f_i$ are linearly independent.
Then $\phi$ is a  linear combination of the $f_i$
(with constant coefficients)
iff the  determinant 
$$
\left|
\begin{array}{cccc}
f_1({\mathbf z}_1) & \cdots & f_1({\mathbf z}_r) &  f_1({\mathbf z}_{r+1})\\
\vdots &&\vdots & \vdots \\
f_r({\mathbf z}_1) & \cdots & f_r({\mathbf z}_r) &   f_r({\mathbf z}_{r+1})\\
\phi({\mathbf z}_1) &\cdots & \phi({\mathbf z}_r) &   \phi({\mathbf z}_{r+1})
\end{array}
\right|
$$
is identically zero as a function on $Z^{r+1}$.

Proof. Since the $f_i$ are linearly independent, 
there are  ${\mathbf z}_1,\dots, {\mathbf z}_r\in Z$
such that 
the upper left $r\times r$ subdeterminant of  is nonzero.
Fix these  ${\mathbf z}_1,\dots, {\mathbf z}_r$ and solve the linear system
$$
\phi({\mathbf z}_i)=\tsum_j\ 
\lambda_j f_j({\mathbf z}_i)\qtq{for $i=1,\dots,r$.}
$$
Replace $\phi$ by $\psi:=\phi-\sum_i \lambda_if_i$
and let ${\mathbf z}_{r+1}$ vary.
Then our determinant is
$$
\left|
\begin{array}{cccc}
f_1({\mathbf z}_1) & \cdots & f_1({\mathbf z}_r) &  f_1({\mathbf z}_{r+1})\\
\vdots &&\vdots & \vdots \\
f_r({\mathbf z}_1) & \cdots & f_r({\mathbf z}_r) &   f_r({\mathbf z}_{r+1})\\
0 &\cdots & 0 &   \psi({\mathbf z}_{r+1})
\end{array}
\right|
$$
and it vanishes iff $\psi({\mathbf z}_{r+1})$
is identically zero.
That is, when $\phi\equiv \sum_j \lambda_j f_j$.\qed
\end{say}

\begin{say}[Linear combination test]\label{lin.comb.test.I}
Let $Z$ be  a  real algebraic variety,  $F$ a vector bundle on $Z$
and  $f_1,\dots, f_r$ linearly independent algebraic sections of $F$.

Given 
$\phi\in C^*\bigl(Z, F\bigr)$, when can we write
$\phi=\sum_i \lambda_if_i$ for some  $\lambda_i\in  \c$?

{\it Answer:} One can either write down a determinental
criterion similar to (\ref{2.good.cases}) or
 reduce this to the Wronskian test
 as follows. 

Consider  $q:\p(F)\to X$, the  space of 1-dimensional
quotients of $F$. Let $u:q^*F\to Q$ be the universal quotient line bundle.
Then $\phi=\sum_i \lambda_if_i$
iff 
$$
u\circ q^*(\phi)=\tsum_i \lambda_i\cdot u\circ q^*(f_i).
$$
The latter is enough to check on a Zariski open cover
of $\p(F)$ where $Q$ is trivial. Thus we recover the
 Wronskian test.\qed
\end{say}

\begin{say}[Membership test for sheaf injections]\label{loc.lift.test}
 Let $Z$ be a real algebraic variety, $E,F$ algebraic vector bundles
and $h: E\to F$  a vector bundle 
map such that $\rank h=\rank E$   on a 
Euclidean dense Zariski open set $Z^0\subset Z$.
Given a section $\phi\in C^*\bigl(Z, F\bigr)$, 
when is it
in the image of $h:C^*\bigl(Z, E\bigr)\to
C^*\bigl(Z, F\bigr)$?

{\it Answer:} Over  $Z^0$, there is a
quotient map
$q:F|_{Z^0}\to Q_{Z^0}$ where $\rank Q_{Z^0}=\rank F-\rank E$ and  
 $\im \bigl(h|_{Z^0}\bigr)=\ker q$.
Then the first lifting condition is:

(1) $q(\phi)=0$. Note that, in the local coordinate functions of $\phi$,
this is a linear condition with polynomial coefficients.

By (\ref{basic.props}.3), $h|_{Z^0}$ has an algebraic splitting
$s:F|_{Z^0}\to E|_{Z^0}$. Note that $s$ is not unique on $E$ but it is
unique on the image of $h$. 
Thus the second  condition says:

(2) The section  $s\bigl(\phi|_{Z^0}\bigr)\in
C^*\bigl(Z^0, E|_{Z^0}\bigr)$ 
extends to a section  of $C^*\bigl(Z, E\bigr)$. 

In order to make this more explicit, choose
local algebraic trivializations of $E$ and of $F$.
Then $\phi$ is given by coordinate functions
$(\phi_1,\dots, \phi_m)$ and $s$ is given by a matrix
$(s_{ij})$ where the $s_{ij}$ are rational functions on $Z$ that are regular on
$Z^0$.  We can bring them to  common denominator and
 write  $s_{ij}=u_{ij}/v$ where
$u_{ij}$ and $v$ are regular on $Z$.
Thus  
$$  
s\bigl(\phi|_{Z^0}\bigr)=
\Bigl(\sum_j s_{1j}\phi_j, \dots,\sum_j s_{nj}\phi_j\Bigr) =
\frac1{v}\Bigl(\sum_j u_{1j}\phi_j, \dots,\sum_j u_{nj}\phi_j\Bigr).
$$  
Let $\Phi$ denote the vector function in the parenthesis on the right.
Then $\Phi\in C^*(Z,E)$ and we are asking if $\Phi/v\in  C^*(Z,E)$
or not. This is exactly one of the question considered in 
  Part 2 of   (\ref{outline}).

Also, if we are considering two function classes
$C_1^*\subset C_2^*$, then (\ref{loc.lift.test}.3) and the assumption
(\ref{function.classes.ass}.6) say that 
 a function $\phi\in C_1^*(Z,F)$ is in the image of
$h:C_2^*\bigl(Z, E\bigr)\to
C_2^*\bigl(Z, F\bigr)$ iff it  is in the image of
$h:C_1^*\bigl(Z, E\bigr)\to
C_1^*\bigl(Z, F\bigr)$.
\qed
\end{say}


\begin{say}[Resolution of singularities]\label{res.say}
Let
${\mathbf D}=\bigl(p:Y\to X, f:p^*E\to F\bigr)$ be  a descent problem.
By Hironaka's theorems (see \cite[Chap.3]{res-book} for a relatively
simple treatment)
there is a resolution of singularities
$r_0:Y'\to Y$. That is, $Y'$ is smooth and $r_0$ is proper and birational
(that is, an isomorphism over a Zariski dense open set).
Note however, that $r_0$ is not surjective in general.
In fact, $r_0(Y')$ is precisely the
Euclidean closure of the smooth locus $Y^{ns}$. 
Thus $Y\setminus r_0(Y')\subset \sing(Y)$. 

We resolve $\sing Y$ to obtain  $r_1: Y'_1\to \sing (Y)$.
The resulting map $Y'\amalg Y'_1\to Y$ is surjective,
except possibly along $\sing(\sing(Y))$. We
can next resolve $\sing(\sing(Y))$ and so on.
After at most $\dim Y$ such steps, we obtain
a smooth, proper morphism $R:Y^R\to Y$
such that $Y^R$ is smooth and $R$ is surjective.
$R$ is an isomorphism over $Y^{ns}$ but
it can have many irreducible components that map
to $\sing(Y)$.

We refer to $Y'\subset Y^R$ as the {\it main components} of
the resolution. 
\end{say}

\begin{prop} \label{irred.fiber.solution} Let
${\mathbf D}=\bigl(p:Y\to X, f:p^*E\to F\bigr)$ be  a descent problem.
Assume that $X,Y$ are irreducible, the generic fiber of $p$
is irreducible, smooth and  $h(x):E(x)\to C^0\bigl(Y_x, F|_{Y_x}\bigr)$
is an injection for general  $x\in p(Y)$.
Then ${\mathbf D}$ has a  scion
$
{\mathbf D}_s=\bigl(p_s:Y_s\to X, f_s:p_s^*E\to F_s\bigr)
$
with  surjective  structure map $r_s:Y_s\to Y$ such that
\begin{enumerate}
\item $Y_s$ is a disjoint union $Y_s^h\amalg Y_s^v$,
\item $\dim p_s\bigl(Y_s^v\bigr)<\dim X$ and
\item $f_s$  is an isomorphism over  $Y_s^h$.
\end{enumerate}
\end{prop}

Proof. Set $n=\rank E$ and 
let $Y^{n+1}_X$ be the  union of the dominant components 
(\ref{r.basic}) of the
$n+1$-fold fiber product of $Y\to X$ with
coordinate projections $\pi_i$. 
Let $\tilde p:Y^{n+1}_X\to X$ be the map given by any of the $p\circ \pi_i$.
Consider the diagonal map
$$
\tilde f: \tilde p^*E\to \tsum_{i=1}^{n+1} \pi_i^*F
$$
which is an injection 
over a Zariski dense Zariski open set $Y^0\subset Y^{n+1}_X$ by assumption.
By (\ref{rel.desc.prob.sefn}),
 these define a scion of ${\mathbf D}$ with  surjective  
structure map.

We want to use  the Local lifting test (\ref{loc.lift.test}) to replace
$\tsum_{i=1}^{n+1} \pi_i^*F$ by $\tilde p^*E$.
For this we need $Y^0$ to be also Euclidean dense. 
To achieve this, we resolve $Y^{n+1}_X$
as in (\ref{res.say}) to get $Y_s$. The main components
give $Y_s^h$ but we  may have introduced some
other components  $Y_s^v$ that map to
$\sing (Y)$. Since the general fiber of $p$
is smooth,  $Y_s^v$ maps to a lower dimensional
subvariety of $X$.
\qed

\begin{prop} \label{irred.fiber.solution.II} Let
${\mathbf D}=\bigl(p:Y\to X, f:p^*E\to F\bigr)$ be  a descent problem.
Assume that $X,Y$ are irreducible and the generic fiber of $p$
is irreducible and  smooth. Then there is a commutative diagram
$$
\begin{array}{ccc}
\bar Y & \stackrel{\tau_Y}{\to} & Y\\
\bar p\downarrow  \hphantom{p} &&   \hphantom{p}\downarrow p\\
\bar X & \stackrel{\tau_X}{\to} & X
\end{array}
$$
where $\tau_X, \tau_Y$ are proper, birational
and there is a quotient bundle $\tau_X^*E\onto\bar E$
such that 
$\bar p^*\tau_X^*E\to  \tau_Y^*F$ factors through
$\bar p^*\bar E$ and
the descent problem
$$
\bar {\mathbf D}=\bigl(\bar p:\bar Y\to \bar X, \bar f:\bar p^*\bar E\to 
\bar F:= \tau_Y^*F\bigr)
$$
satisfies the assumptions of (\ref{irred.fiber.solution}).
That is,
  $\bar f(x):\bar E(x)\to 
C^0\bigl(\bar Y_{x}, \tau_Y^*F|_{\bar Y_{x}}\bigr)$
is an injection for general  $x\in \bar p(\bar Y)$.

Moreover, if a finite group $G$ acts on ${\mathbf D}$ then
we can choose  $\bar {\mathbf D}$
such that the $G$-action lifts to $\bar {\mathbf D}$.
\end{prop}

(Note that, as shown by (\ref{only.nash.exmp}), the conclusions can fail
if the general fibers of $p$ are not irreducible.)
\medskip

Proof. 
Complexify $p:Y\to X$ to get a
complex proper morphism $p_{\c}: Y_{\c}\to X_{\c}$  and set
$$
E'_{\c}:= \im\bigl[E_{\c}\to  
\bigl(p_{\c}\bigr)_* F_{\c}\bigr].
$$
Let $x\in p(Y)$ be a general point. Then $Y_x$ is irreducible
and the real points $Y_x$ are Zariski dense
in the complex fiber  $\bigl(Y_{\c}\bigr)_x$.
Thus $H^0\bigl(\bigl(Y_{\c}\bigr)_x, F_{\c}\bigr)=
H^0\bigl( Y_x, F\bigr)$. 

So far $E'_{\c}$ is only a coherent sheaf which is
a quotient of $E_{\c}$. Using (\ref{loc.free.image}) and then
 (\ref{irred.fiber.solution}),
we obtain $\tau_X:\bar X\to X$ as desired.\qed

\begin{say}\label{loc.free.image} Let $X$ be an irreducible variety
 $q:E\to E'$ a map of vector bundles on $X$.
In general we can not write $q$ as a composite of a surjection
of vector bundles followed by an injection, but
the following construction shows how to achieve this
after modifying $X$.

Let $\grass(d,E)\to X$ be the universal Grassmann bundle of 
rank $d$ quotients of $E$ where $d$ is the  rank of $q$ at a general point. 
At a general point $x\in X$, $q(x):E(x)\onto \im q(x)\subset E'(x)$
 is such a quotient.
Thus $q$ gives a rational map
$X\map \grass(d,E)$, defined on a Zariski dense 
Zariski open subset. Let $\bar X\subset \grass(d,E)$
 denote the
closure of its image  and $\tau_X:\bar X\to X$
 the projection.
Then  $\tau_X$ is a proper birational morphism
and we have a decomposition
 $$
\tau_X^*q: \tau_X^*E\stackrel{s}{\onto} \bar E\stackrel{j}{\into}  \tau_X^*E'
$$
where $\bar E$  is a vector bundle of rank $d$ on $\bar X$,
$s$ is a rank $d$ surjection everywhere and $j$
is a rank $d$ injection on a Zariski dense 
Zariski open subset.
\end{say}

\subsection{Proof of the main algebraic theorem}{\ }

In order to answer  Question \ref{rq1}
 in general, we try to create a situation
where   (\ref{irred.fiber.solution.II}) applies. 

First, using (\ref{res.say}) we may assume that $Y$ is smooth.
Next
take the Stein factorization
$Y\to W\to X$; that is, $W\to X$ is finite and
all the fibers of $Y\to W$ are connected
(hence general fibers are irreducible).

After some modifications,
 (\ref{irred.fiber.solution}) applies to
$Y\to W$, thus we are reduced to comparing
$C^0(W, p_W^*E)$ and $C^0(X,E)$.   

This is easy if $W\to X$ is Galois, since then
the sections of $p_W^*E$ that are invariant under the Galois group
descend to sections of $E$.

If $p:W\to X$ is a finite morphism of (smooth or at least normal) 
varieties over $\c$,
the usual solution would be to take the Galois closure of the 
field extension $\c(W)/\c(X)$ and 
let $W^{Gal}\to X$ be the normalization of $X$ in it.
Then the Galois group $G$ acts on $W^{Gal}\to X$
and the action is transitive on every fiber.

This does not work for real varieties since in general,
$W^{Gal}$ has no real points. 
(For instance, take $X=\r$ and let $W\subset \r^2$ be any curve
given by an irreducible equation of the form 
$y^m=f(x)$. If $m=2$ then $W/X$ is Galois but for
$m\geq 3$   the  Galois closure $W^{Gal}$ has no real points.) 
Some other problems are illustrated by the next example.

\begin{exmp}\label{only.nash.exmp} Let  $W\subset \r^2$ be defined by
$(y^5-5y=x)$ with $p:W\to \r^1_x=:X$ the projection.
Set $E=\c^4_X$ and $F=\c_W$ with 
$f:p^*E\to F$ given by  $f(\psi(x)e_i)=y^i\psi(x)|_W$.

Note that $p$ has degree 5 as a map of (complex) Riemann surfaces,
but $p^{-1}(x)$ consists of 3 points for $-1<x<1$
and of 1 point if $|x|>1$. 
Therefore,  the kernel of 
$f\circ p^*(x):\c^4=E(x)\to C^0\bigl(W_x, F|_{W_x}\bigr)$
has rank 1 if $-1<x<1$ and rank 3 if $|x|>1$.
Thus $\ker \bigl(f\circ p^*\bigr)\subset E$ is a rank 1 subbundle
on the interval  $-1<x<1$ and a  rank 3  subbundle
on the intervals  $|x|>1$.

These kernels depend only on some of the 5 roots of
$y^5-5y=x$, hence they are  semialgebraic  subbundles
but not real algebraic subbundles.
\end{exmp}

As a replacement of the  Galois closure $W^{Gal}$, we next introduce
a series of varieties  $W^{(m)}_X\to X$.
The $W^{(m)}_X$  are usually reducible,
the symmetric group $S_m$ acts on them, but
 the  $S_m$-action is usually  not transitive on every fiber.
Nonetheless,  all  the $W^{(m)}_X$ together  provide a suitable analog of the
Galois closure.

\begin{defn} \label{W(m).defn}
Let $s:W\to X$ be a finite morphism of (possibly reducible) varieties and
 $X^0\subset X$  the largest Zariski open subset over which $p$ is smooth.

Consider the $m$-fold fiber product
$W^m_X:=W\times_X\cdots\times_XW$ with coordinate projections
$\pi_i:W^m_X\to W$. For every $i\neq j$,
let $\Delta_{ij}\subset W^m_X$ be the preimage of the diagonal
 $\Delta\subset W\times_XW$ under the map $(\pi_i,\pi_j)$.
Let $W^{(m)}_X\subset W^m_X$ be the 
union of the dominant components in the 
closure of
$W^m_X\setminus\cup_{i\neq j}\Delta_{ij}$
with projection $s^{(m)}:W^{(m)}_X\to X$.
The symmetric group $S_m$ acts on $W^{(m)}_X$ by permuting the factors.

If $x\in X^0$ then
$\bigl(s^{(m)}\bigr)^{-1}(x)$ consists of
ordered $m$-element subsets of $s^{-1}(x)$.
Thus $\bigl(s^{(m)}\bigr)^{-1}(x)$ is empty if
$|s^{-1}(x)|<m$ and
$S_m$ acts transitively on $\bigl(s^{(m)}\bigr)^{-1}(x)$ if
$|s^{-1}(x)|=m$.
If $\bigl|s^{-1}(x)\bigr|>m$ then 
$S_m$ does not act  transitively on
$\bigl(s^{(m)}\bigr)^{-1}(x)$.
We obtain a decreasing sequence of semialgebraic subsets
$$
s^{(1)}\bigl(W^{(1)}_X\bigr)\supset s^{(2)}\bigl(W^{(2)}_X\bigr)\supset \cdots.
$$
Set 
$$
X_{W,m}^0:=X^0\cap  \Bigl(s^{(m)}\bigl(W^{(m)}_X\bigr)\setminus
s^{(m+1)}\bigl(W^{(m+1)}_X\bigr)\Bigr).
$$
The $X_{W,m}^0$ are disjoint, 
$\bigcup_mX_{W,m}^0$ is a Euclidean dense semialgebraic open subset
of $p(Y)\cap X^0$ and the $S_m$-action is transitive on the fibers of
$s^{(m)}$ that lie over $X_{W,m}^0$. 
Thus $s^{(m)}:W^{(m)}\to X$ behaves like a Galois extension
over $X_{W,m}^0$ and together the $X_{W,m}^0$
cover most of $X$.

Let now $p:Y\to X$ be a proper morphism of  (possibly reducible) 
normal varieties
with Stein factorization $p:Y\stackrel{q}{\to} W\stackrel{s}{\to} X$.
Let $Y^m_X$ denote the $m$-fold fiber product $Y\times_X\cdots\times_XY$
 with coordinate projections
$\pi_i:Y^m_X\to Y$.

Let  $Y^{(m)}_X\subset Y^m_X$ 
denote the  dominant parts of the 
preimage of  $W^{(m)}_X$ under the natural map
$q^m:Y^m_X\to W^m_X$ with projection $p^{(m)}:Y^{(m)}_X\to X$.
Note that, for general $x\in X$, 
 $\bigl(p^{(m)}\bigr)^{-1}(x)$ is empty if
$p^{-1}(x)$  has fewer than $m$ irreducible components
and
$S_m$ acts transitively on the  irreducible components of 
$\bigl(p^{(m)}\bigr)^{-1}(x)$ if
$p^{-1}(x)$  has exactly $m$ irreducible components.
Thus we obtain a decreasing sequence of semialgebraic subsets
$p^{(1)}\bigl(Y^{(1)}_X\bigr)\supset p^{(2)}\bigl(Y^{(2)}_X\bigr)\supset \cdots $.

Let $F$ be a vector bundle on $Y$.
Then $\oplus_i \pi_i^*F$ is a   vector bundle on  $Y^m_X$.
Its restriction to  $Y^{(m)}_X$ is denoted by $F^{(m)}$. 

Note that the $S_m$-action on  $Y^{(m)}_X$ naturally lifts to
an  $S_m$-action on  $F^{(m)}$. 
If $E$ is a   vector bundle on  $X$ and $f:p^*E\to F$ is a
  vector bundle map then we get an $S_m$-invariant
 vector bundle map 
$f^{(m)}:\bigl(p^{(m)}\bigr)^*E\to F^{(m)}$. 
For each $m$ we get a scion of  ${\mathbf D}$
$$
{\mathbf D}^{(m)}:=\bigl(p^{(m)}:Y^{(m)}_X\to X, 
f^{(m)}:\bigl(p^{(m)}\bigr)^*E\to F^{(m)}\bigr).
$$
Below, we will use all the ${\mathbf D}^{(m)}$ together to get a
 scion with Galois-like properties.
\end{defn}

\begin{say}[Proof of (\ref{main.thm.induct})]\label{pf.of.main.thm.induct}

If  ${\mathbf D}_s$ is a scion of  ${\mathbf D}$
with  surjective  structure map $r_s:Y_s\to Y$,
then (\ref{main.thm.induct}.1) $\Leftrightarrow$
(\ref{main.thm.induct}.2) by definition
and  (\ref{main.thm.induct}.2) $\Rightarrow$
(\ref{main.thm.induct}.3)
holds for any choice of $Z$.

Assume next that we have a candidate for
 ${\mathbf D}_s$ and $Z$ such that.
 How do we check (\ref{main.thm.induct}.3) 
$\Rightarrow$ (\ref{main.thm.induct}.2)?

Pick $\Phi_s\in C^*\bigl(Y_s, F_s)$ and assume that
there is a section $\phi_Z\in C^*\bigl(Z, E|_{Z}\bigr)$
whose pull-back to $Y_Z$ equals the restriction
of $\Phi_s$. 
By (\ref{basic.props}), we can lift $\phi_Z$ to a section
$\phi_X \in C^*\bigl(X, E\bigr)$.
Consider next  
$$
\Psi_s:=\Phi_s-f_s\bigl(p_s^*\phi_X\bigr)\in  
C^*\bigl(Y_s, F_s\bigr).
$$  
We are done if we can write $\Psi_s=f_s\circ p_s^*(\psi_X)$
for some $\psi_X \in C^*\bigl(X, E\bigr)$.
 
By assumption, $\Psi_s$ satisfies the  finite set test (\ref{fin.set.test.defn})
  but the
 improvement is that $\Psi_s$  vanishes on $Y_Z$.
As we saw already in  (\ref{zero.liit.lem}), this can make
the problem much easier.
We deal with this case  in (\ref{red.step.main}).

Note that by \cite{whi}, we can choose
$\phi_X$ to be real analytic away from $Z$ and the rest of the construction
preserves differentiability properties. Thus
 (\ref{r.main.cors}.2) holds once the rest of the  argument is worked out.
 \qed
\end{say}

\begin{prop}\label{red.step.main}
 Let ${\mathbf D}=\bigl(p:Y\to X, f:p^*E\to F\bigr)$
 be a  descent problem.
Then there is a  closed
 algebraic subvariety $Z\subset X$ with $\dim Z<\dim X$ and 
 a  scion 
${\mathbf D}_s=\bigl(p_s:Y_s\to X, f_s:p_s^*E\to F_s\bigr)$
with  surjective  structure map $r_s:Y_s\to Y$
such that the following holds.

Let $\psi_s\in C^0(Y_s, F_s)$ be a section  that
 vanishes on  $p_s^{-1}( Z)$
and satisfies the   finite set test (\ref{fin.set.test.defn}).
Then there is a $\psi_X\in C^0(X,E)$ such that
$\psi_X$ vanishes on  $Z$ and 
$\psi_s=f_s\circ p_s^*(\psi_X)$.
\end{prop}

Proof. 
We may harmlessly assume that $p(Y)$ is Zariski dense in $X$.
Using (\ref{res.say}) we may also assume that $Y$ is smooth.

After we construct  ${\mathbf D}_s$,
the plan 
is to make sure that $Z$ contains all of its  ``singular'' points. 
In the original setting of Question \ref{rq1},
$Z$ was the set where the map
$(f_1,\dots, f_r):\c^r\to \c$ has rank 0.
In the general case, we need to include points over which
$f_s$ drops rank and also points over which $p_s$ drops rank.

During the proof we gradually add more and more
irreducible components to $Z$.
To start with, we  add to $Z$ the lower dimensional
irreducible components of $X$, the locus where $X$ is not normal
and  the (Zariski closures of)  the $p(Y_i)$ where
$Y_i\subset Y$ is an irreducible component
that does not dominate any of the  maximal dimensional 
irreducible components  of $X$.
We can thus assume that $X$ is irreducible and
every  irreducible component of $Y$ dominates $X$.

Take  the Stein factorization $p:Y\stackrel{q}{\to} W\stackrel{s}{\to} X$
and set $M=\deg(W/X)$.
For each $1\leq m\leq M$, consider 
the following diagram
$$
\begin{array}{rclclcc}
\bigl(\bar q^{(m)}\bigr)^*\bar E^{(m)} & \cong & 
\bar F^{(m)} &  & F^{(m)} &  & F\\
&& \ \downarrow  &&\ \downarrow  &&\downarrow  \\
\bigl(t_W^{(m)}\circ s_W^{(m)}\bigr)^*E\onto \bar E^{(m)} & & \bar Y^{(m)}_X & 
\stackrel{t_Y^{(m)}}{\to} & Y^{(m)}_X & \stackrel{\pi_i^{(m)}}{\to} & Y\\
&\searrow & \ \downarrow\bar q^{(m)}
  &&\ \downarrow q^{(m)} &&\hphantom{p}\downarrow p  \\
& & \bar W^{(m)} & 
\stackrel{t_W^{(m)}}{\to} & W^{(m)} &  
\stackrel{s^{(m)}}{\to}  & X
\end{array}
\eqno{(\ref{red.step.main}.m)}
$$
where $ W^{(m)}$ and its column is constructed in
(\ref{W(m).defn}) and out of this $\bar W^{(m)}$, its column 
and the vector bundle $\bar E^{(m)}$ are constructed in
(\ref{irred.fiber.solution.II}).
Note that the symmetric group $S_m$ acts on the whole diagram.

The ${\mathbf D}_s$ we use will be the  disjoint union of the scions
$$
\bar{\mathbf D}_s^{(m)}:=\bigl(\bar p^{(m)}: \bar Y^{(m)}_X\to X, 
\bar f^{(m)}:\bigl(\bar p^{(m)}\bigr)^*E\to \bar F^{(m)}\bigr)
\qtq{for $m=1,\dots, M$.}
$$

By enlarging $Z$ if necessary, we may assume that 
 $Y^{(m)}_X\to X$ is smooth over $X\setminus Z$ and
 each $ t_W^{(m)}$ is an isomorphism  over $X\setminus Z$.
Note that, for every $m$,
$$
X^0_m:=p^{(m)}\bigl(Y^{(m)}_X\bigr)\setminus 
\Bigl(Z\cup p^{(m+1)}\bigl(Y^{(m+1)}_X\bigr)\Bigr)
\subset X
$$
 is an open  semialgebraic subset
of  $X\setminus Z$ whose  boundary is in $Z$.
Furthermore,  $p(Y)\setminus Z$ is the disjoint union of the  $X^0_m$
and  the fiber $Y_x$ has exactly $m$ irreducible components
if $x\in X^0_m$.

Let $\Psi_s\in C^0(Y_s, F_s)$ be a section  that
 vanishes on  $p_s^{-1}( Z)$.
We can then uniquely write $\Psi_s=\sum_m \Psi_s^{(m)}$ such that each
$\Psi_s^{(m)}$ vanishes on $Y_s\setminus p_s^{-1}\bigl(X^0_m\bigr)$. 
Moreover,   $\Psi_s$ satisfies the  finite set test (\ref{fin.set.test.defn})
iff all the $\Psi_s^{(m)}$ satisfy it. 

Thus it is sufficient to prove that each
 $\Psi_s^{(m)}$ is the pull-back of a
section $\psi_X^{(m)}\in C^*(X,E)$ that
vanishes on $X\setminus X^0_m$.
For each $m$ we use the corresponding diagram (\ref{red.step.main}$_m$).

Each $\Psi_s^{(m)}$ lifts to a section 
 $\bar\Psi_s^{(m)}$ of  $\bigl(\bar q^{(m)}\bigr)^*\bar E^{(m)}$
that satisfies the pull-back conditions 
for $\bar Y^{(m)}\to \bar W^{(m)}$.
Thus $\bar\Psi_s^{(m)}$ is the pull-back of a section $\bar\Psi_W^{(m)}$ of
$\bar E^{(m)}$. By construction, $\bar\Psi_W^{(m)}$
is  $S_m$-invariant  and it vanishes outside 
$\bigl(t^{(m)}\circ s^{(m)}\bigr)^{-1}\bigl(X^0_m\bigr)$. 
Using a splitting of 
$\bigl(s_W^{(m)}t_W^{(m)}\bigr)^*E\onto \bar E^{(m)}$
we can think of $\bar\Psi_W^{(m)}$ as an  $S_m$-invariant section of
$\bigl(t_W^{(m)}\circ s_W^{(m)}\bigr)^*E$. 
By the choice of $Z$, $t^{(m)}$ is an isomorphism
over $X^0_m$, hence   $\bar\Psi_W^{(m)}$
descends to an  $S_m$-invariant  section 
$\Psi_W^{(m)}$ of $\bigl(s_W^{(m)}\bigr)^*E$
that vanishes outside $\bigl(s_W^{(m)}\bigr)^{-1}\bigl(X^0_m\bigr)$.
 Therefore, 
by (\ref{pull-back.test}.2),   $\Psi_W^{(m)}$ descends to a   section
$\psi_X^{(m)} \in C^0\bigl(X, E\bigr)$
that vanishes on $X\setminus X^0_m$.
\qed

\subsection{Semialgebraic, real and $p$-adic analytic cases}

\begin{say}[Real analytic case]\label{realan.say}
It is natural to ask Question \ref{rq1} when the $f_i$ are 
real analytic functions and $\r^n$ is replaced by an arbitrary
 real analytic variety.
As before, we think of $X$ as the real points of  a  complex Stein space
  $X_{\c}$ that is  defined 
by real equations. 
Our proofs work without changes
for descent problems
 ${\mathbf D}=\bigl(p:Y\to X, f:p^*E\to F\bigr)$
where $Y$ and $f$ are {\it relatively algebraic} over $X$.

By definition, this means that $Y$ is the set of 
 real points of a closed (reduced but possibly reducible)
complex analytic  subspace of some
$X_{\c}\times \c\p^N$ and that $f$ is
 assumed algebraic in the
$\c\p^N$-variables.

This definition may not seem the most natural, but it 
is exactly the setting needed to answer Question  \ref{rq1}
if the $f_i$ are  real analytic functions on a real analytic space.
\end{say}

\begin{say}[Semialgebraic case]\label{semialg.say}
It is straightforward to consider semialgebraic
descent problems
 ${\mathbf D}=\bigl(p:Y\to X, f:p^*E\to F\bigr)$
where $X,Y$ are semialgebraic sets, $E,F$ are
 semialgebraic vector bundles and $p,f$ are
 semialgebraic maps. 
(See \cite[Chap.2]{bcr} for basic results and definitions.) 
It is not hard to go through the proofs and
see that everything generalizes to the semialgebraic case.

In fact, some of the constructions could be simplified
since one can break up any descent problem ${\mathbf D}$
into a union of descent problems ${\mathbf D}_i$
such that each $Y_i\to X_i$ is topologically a product
over the interior of $X_i$. This would allow one to make some
non-canonical choices to simplify the 
construction of the diagrams (\ref{red.step.main}.m).

It may be, however, worthwhile to note that
one can directly reduce the  semialgebraic version to the
real algebraic one as follows.

Note first that in the  semialgebraic  setting
 it is natural to replace a real algebraic descent problem
 ${\mathbf D}=\bigl(p:Y\to X, f:p^*E\to F\bigr)$
by its 
{\it semialgebraic reduction}
$\operatorname{sa-red}({\mathbf D})
:=\bigl(p:Y\to p(Y), f:p^*\bigl(E|_{p(Y)}\bigr)\to F\bigr)$.

We claim that for every semialgebraic
descent problem
 ${\mathbf D}$ there is a 
proper surjection $r:Y_s\to Y$ such that the corresponding scion
$r^*{\mathbf D}$ is semialgebraically isomorphic to 
the semialgebraic reduction of a
real algebraic descent problem.

To see this, first, we can replace the semialgebraic  $X$
by a real algebraic variety $X^a$ that contains it and
extend $E$ to  semialgebraic vector bundle over $X^a$.
Not all  semialgebraic vector bundles are algebraic,
but we can realize $E$ as a  semialgebraic subbundle 
of a trivial bundle $\c^M$. 
This in turn gives a  semialgebraic embedding of
$X$ into $X\times \grass(\rank E,M)$. Over the image, $E$ is
the restriction of the algebraic universal bundle on $\grass(\rank E,M)$.
Thus, up to replacing $X$ by the Zariski closure of its image,
we may assume that $X$ and $E$ are both algebraic.
Replacing $Y$ by the graph of $p$  in $Y\times X$,
we may assume that $p$ is algebraic. 
Next write $Y$ as the image of a real algebraic variety.
We obtain a  scion where now
$p:X\to Y, E, F$ are all algebraic.
To make $f$ algebraic, we use that $f$ defines a
semialgebraic  section of 
$\p\bigl(\shom_X(p^*E,F)\bigr)\to Y$. 
Thus, after 
replacing $Y$ by the Zariski closure of its image in 
$\p\bigl(\shom_X(p^*E,F)\bigr)\to Y$,
we obtain an algebraic scion with  surjective  structure map.
\end{say}

\begin{say}[$p$-adic case]\label{p-adic.say}
One can also consider  Question \ref{rq1}
in the $p$-adic case and the proofs work without any changes.
In fact, if we start with polynomials
$f_i\in \q[x_1,\dots, x_n]$ then in Theorem
 \ref{main.thm.v2} 
it does not matter whether we want to work over $\r$
or $\q_p$; we construct the same
descent problems. It is only in checking the  finite set test 
(\ref{fin.set.test.defn})
that the field needs to be taken into account:
if we work over $\r$, we need to check the condition for
fibers over all real points, if we work over $\q_p$,
 we need to check the condition for
fibers over all $p$-adic points.
\end{say}

\begin{ack} We thank M.~Hochster for communicating his
unpublished example (\ref{rpwt}.4).
We are grateful to B.~Klartag and A.~Naor for 
bringing Michael's theorem to our attention at a workshop organized by the
 American Institute of Mathematics (AIM), to which we are also grateful. 
 Our earlier proof of  (\ref{35.8}) was unnecessarily complicated.
We thank
 H.~Brenner, A.~Isarel,  K.~Luli, R.~Narasimhan, A.~N\'emethi and T.~Szamuely
 for helpful conversations  and
 F.~Wroblewski
for TeXing several sections of this paper.

Partial financial support for CF  was provided by  the NSF under grant number 
DMS-0901040 and by the ONR under grant number  N00014-08-1-0678. 
Partial financial support for JK  was provided by  the NSF under grant number 
DMS-0758275.
\end{ack}


\noindent Princeton University, Princeton NJ 08544-1000

{\begin{verbatim}cf@math.princeton.edu\end{verbatim}}

{\begin{verbatim}kollar@math.princeton.edu\end{verbatim}}


\providecommand{\bysame}{\leavevmode\hbox to3em{\hrulefill}\thinspace}
\providecommand{\MR}{\relax\ifhmode\unskip\space\fi MR }
\providecommand{\MRhref}[2]{%
  \href{http://www.ams.org/mathscinet-getitem?mr=#1}{#2}
}
\providecommand{\href}[2]{#2}

\end{document}